\pdfoutput=1
\documentclass[a4paper,reqno]{amsart}
\usepackage[utf8]{inputenc}
\usepackage{color}
\usepackage{prettyref}
\usepackage{textcomp}
\usepackage{mathrsfs}
\usepackage{mathtools}
\usepackage{dsfont}
\usepackage{amsbsy}
\usepackage{amstext}
\usepackage{amsthm}
\usepackage{amssymb}
\usepackage{stmaryrd}
\usepackage{microtype}
\usepackage[unicode=true,
 bookmarks=false,
 breaklinks=true,pdfborder={0 0 0},pdfborderstyle={},backref=false,colorlinks=true]
 {hyperref}
\hypersetup{
 pdfstartview=FitV,linkcolor=black,citecolor=black}

\makeatletter

\pdfpageheight\paperheight
\pdfpagewidth\paperwidth

 \numberwithin{equation}{section}
\numberwithin{figure}{section}
\theoremstyle{plain}
\newtheorem{thm}{\protect\theoremname}[section]
\theoremstyle{remark}
\newtheorem{rem}[thm]{\protect\remarkname}
\theoremstyle{definition}
\newtheorem{defn}[thm]{\protect\definitionname}
\theoremstyle{plain}
\newtheorem{lem}[thm]{\protect\lemmaname}
\theoremstyle{plain}
\newtheorem{cor}[thm]{\protect\corollaryname}
\theoremstyle{remark}
\newtheorem{notation}[thm]{\protect\notationname}
\theoremstyle{plain}
\newtheorem{prop}[thm]{\protect\propositionname} 
\usepackage{bbm}
\usepackage{pgfplots}
\usepackage[scr=rsfs,cal=boondox]{mathalfa}
\usepgfplotslibrary{colormaps,fillbetween}
\pgfplotsset{compat=1.15}
\usepackage{mathrsfs}
\usetikzlibrary{arrows}
\usetikzlibrary{patterns}
\usepackage[OT2,OT1]{fontenc}

\usepackage{txfonts}
\usepackage{textcomp}
\usepackage{xcolor}
\usepackage{dsfont}
\usepackage{graphicx}
\usepackage{caption}
\usepackage{subcaption}
\usepackage{float}
\usepackage{tikz}
\usepackage[foot]{amsaddr}
\usepackage{enumitem}

\renewcommand{\rho}{\varrho}
\renewcommand{\epsilon}{\varepsilon}
\newcommand{\1}{\mathbbm{1}}

\newcommand{\Q}{\mathbf{Q}}
\renewcommand{\d}{\;\mathrm{d}}
\DeclareMathOperator{\supp}{supp}

\newcommand{\uAO}{\overline{\mathbf{ord}}} 
\newcommand{\lAO}{\underline{\mathbf{ord}}} 
\newcommand{\AO}{\mathbf{ord}}

\renewcommand{\tilde}{\widetilde}
\renewcommand{\emptyset}{\varnothing}
\renewcommand{\phi}{\varphi}
\renewcommand{\epsilon}{\varepsilon}

\def\Q{\textbf{Q}}
\def\r{\varrho}
\def\rr{\hat{\varrho}}

\def\b{\mathcal{b}}

\newcommand{\s}{\alpha}

\DeclareMathOperator{\spann}{span}
\DeclareMathOperator{\card}{card}

\renewcommand{\tilde}{\widetilde}
\renewcommand{\emptyset}{\varnothing}
\renewcommand{\phi}{\varphi}
\renewcommand{\a}{a}
\def\N{\mathbb N}
\def\d{\;\mathrm d}
\def\R{\mathbb{R}}

\def\Q{\mathcal{Q}}
\def\B{B_{\s}^{p}}
\def\BB{\mathscr{B}W_{0}^{\s,p}}
\def\PP{\mathbb{P}}

\def\J{\mathfrak{J}}

\AtBeginDocument{
\DeclareFieldFormat[article]{title}{{#1}}
\DeclareFieldFormat[unpublished]{title}{{#1}}
\DeclareFieldFormat[proceedings]{title}{{#1}}
\DeclareFieldFormat[incollection]{title}{{#1}}
\DeclareFieldFormat[unpublished]{title}{{\em #1}}
\renewbibmacro{in:}{}
\DeclareFieldFormat[article]{pages}{#1}
}
\def\cprime{$'$}
\newrefformat{prop}{Proposition~\ref{#1}}
\newrefformat{fact}{Fact~\ref{#1}}
\newrefformat{exa}{Example~\ref{#1}}
\newrefformat{cor}{Corollary~\ref{#1}}
\newrefformat{assu}{Assumption~\ref{#1}}
\newrefformat{subsec}{Section~\ref{#1}}
\newrefformat{sec}{Section~\ref{#1}}
\newrefformat{def}{Definition~\ref{#1}}
\newrefformat{enu}{(\ref{#1})}
\newrefformat{rem}{Remark \ref{#1}}

\makeatother

\usepackage[style=authoryear,backend=bibtex,style=alphabetic,  backref=false, giveninits=true,citestyle=alphabetic, natbib=false, url=false, isbn=false,doi=true,eprint=false,maxbibnames=99]{biblatex}
\providecommand{\corollaryname}{Corollary}
\providecommand{\definitionname}{Definition}
\providecommand{\lemmaname}{Lemma}
\providecommand{\notationname}{Notation}
\providecommand{\propositionname}{Proposition}
\providecommand{\remarkname}{Remark}
\providecommand{\theoremname}{Theorem}

\begin{document}
\title[approximation order for Sobolev embeddings in euclidian measure spaces]{Approximation order of Kolmogorov, Gel{\cprime}fand, and linear
widths for Sobolev embeddings in euclidian measure spaces}
\author{Marc Kesseböhmer}
\email{mhk@uni-bremen.de}
\author{Linus Wiegmann}
\email{linus3@uni-bremen.de}
\address{Institute for Dynamical Systems, Faculty 3 – Mathematics und Computer
Science, University of Bremen, Bibliothekstr. 5, 28359 Bremen, Germany}
\begin{abstract}
In this paper we completely solve the problem of finding the (upper)
approximation order with respect to the Kolmogorov, Gel{\cprime}fand,
and linear widths for the embedding of the Sobolev spaces $W^{\s,p}$
and $W_{0}^{\s,p}$ in the euclidian measure spaces $L_{\nu}^{q}$
for an arbitrary Borel probability measure $\nu$ with support contained
in the open $m$-dimensional unit cube and for all possible choices
of $1\leq p,q\leq\infty$. We will determine the exact values for
the various upper approximation orders in terms of the $L^{q}$-spectrum
of $\nu$ only and finally give sufficient conditions imposed on the
regularity of the $L^{q}$-spectrum for the approximation orders to
exist. We also elucidate some intrinsic connections between the concept
of approximation order and the fractal geometric notion of the upper
and lower Minkowski dimension of the support of $\nu$.
\end{abstract}

\keywords{Kolmogorov widths, Gel{\cprime}fand widths, linear widths, Sobolev
spaces, $L^{q}$-spectrum, Minkowski dimension, partition function,
coarse multifractal formalism.}
\subjclass[2000]{46A32; 35P20; 42B35; 31B30; 28A80}

\maketitle

\section{Introduction and statement of main results}

In this paper we investigate the Kolmogorov, Gel{\cprime}fand, and
linear upper and lower approximation order of the embedding of the
unit ball of the Sobolev space $W_{0}^{\s,p}$, respectively $W^{\s,p}$,
into $L_{\nu}^{q}$, where $\nu$ is some Borel probability measure
with support contained in the unit cube. If $\nu$ is the restriction
of the Lebesgue measure, then the approximation orders have long been
well understood (cf\@. \cite{MR1074378}). Analogous statements for
smooth compact Riemannian manifolds have been found in more recent
times by Geller and Peseneson in \cite{MR3204029,MR3250051,MR3487230}.
If the measure has a singular part but is not the restriction of the
Hausdorff measure to a manifold, then only rough upper bounds have
been obtained by Birman/Solomjak and Borzov in \cite{Borzov1971,MR0281860,MR0301568,MR0209733,MR4444736};
first improvements in terms of the $L^{q}$-spectrum have been provided
recently by Niemann and the first author in \cite{MR4444736} but
only for the case $1<p\leq q<\infty$. In this paper we close the
gap and determine the upper approximation order for arbitrary measures
for all choices $1\leq p,q\leq\infty$ precisely in terms of the $L^{q}$-spectrum
and give easy checkable regularity conditions on the $L^{q}$-spectrum
of $\nu$ to guarantee the existence of the approximation order.

To determine the quantities of interest we make use of the method
of discretisation combined with some recent work on optimal partitions
elaborated in the context of Kre\u{\i}n–Feller operators in \cite{KN2022,KN2022b}.
As pointed out by Tikhomirov \cite{MR1074378} the method of discretisation
has first been introduced by Ismagilov \cite{MR0230117,MR0407509}
and was further developed by Ma\u{\i}orov \cite{MR0402349}, Ka\v sin
\cite{Kas77} and Glu\v skin \cite{MR687610}. Let us start with
introducing the relevant Sobolev spaces.

\subsection{Sobolev spaces}

Let us start with some basic notations to introduce the relevant Sobolev
spaces. Let $\R^{m}$ denote the $m$-dimensional euclidean space,
$m\in\N$. For a multi-index $k=\left(k_{1},\dots,k_{m}\right)\in\N_{\text{0}}^{m}$
we define $\left\vert k\right\vert \coloneqq\sum_{i=1}^{m}k_{i}$
and $x^{k}\coloneqq\prod_{i=1}^{m}x_{i}^{k_{i}}$. We consider the
domain $\mathring{\Q}$ denoting the interior of the half-open unit
cube $\Q=\left(0,1\right]^{m}$ and for $p\in\R_{\geq1}$, we let
$L^{p}\coloneqq L^{p}\left(\Q\right)=L_{\Lambda}^{p}\left(\Q\right)$
denote the set of real-valued $p$-integrable functions on $\mathring{\Q}$
with respect to the Lebesgue measure $\Lambda$ restricted to $\mathring{\Q}$.
The\emph{ Sobolev space}
\[
W^{\s,p}\left(\Q\right)\coloneqq\left\{ f\in L^{p}\left(\Q\right):\forall k\in\N_{0}^{m}\text{ with }\left\vert k\right\vert \leq\s,D^{k}f\in L^{p}\left(\Q\right)\right\} 
\]
(see e.\,g\@. \cite{MR1125990} and \cite{1961RuMa}) is defined
to be the set of all functions $f\in L^{p}\left(\Q\right)$, for which
the weak derivatives $D^{k}f\coloneqq\partial^{|k|}/\left(\partial_{x_{1}}^{k_{1}}\cdot\cdot\cdot\partial_{x_{m}}^{k_{m}}\right)f$
up to order $\s\in\N$ lie in $L^{p}\left(\Q\right)$, equipped with
the norm 
\[
\left\Vert f\right\Vert _{W^{\s,p}\left(\Q\right)}\coloneqq\left\Vert f\right\Vert _{L^{p}\left(\Q\right)}+\left\Vert f\right\Vert _{L^{\s,p}\left(\Q\right)},
\]
where we set $\left\Vert f\right\Vert _{L^{\s,p}\left(\Q\right)}\coloneqq\left(\int_{\Q}\left|\nabla_{\s}f\right|^{p}\d\Lambda\right)^{1/p}$
with $\left|\nabla_{\s}f\right|\coloneqq\left(\sum_{\left|k\right|=\s}\left|D^{k}f\right|^{2}\right)^{1/2}$.
We let $W_{0}^{\s,p}\left(\Q\right)$ denote the completion of $\mathcal{C}_{c}^{\infty}\left(\mathring{\Q}\right)$
with respect to $\left\Vert \,\cdot\,\right\Vert _{W^{\s,p}\left(\Q\right)}$,
where $\mathcal{C}_{c}^{\infty}\left(\mathring{\Q}\right)$ denotes
the set of all infinitely differentiable functions with compact support
in $\mathring{\Q}.$ For any half-open cube $Q\subset\Q$, which —throughout
the paper— are assumed to have edges parallel to the coordinate axes,
we have by Friedrichs' inequality, that the space $W_{0}^{\s,p}(Q)$
carries the equivalent norm given by $\left\Vert \,\cdot\,\right\Vert _{L^{\s,p}(Q)}$.
For the set of continuous function from $B\subset\R^{m}$ to $\R$
we write $\mathcal{C}\left(B\right).$ If 
\begin{equation}
1\leq p,q\leq\infty\;\text{and }\;\rr\coloneqq\s-m/p>0\tag{{\ensuremath{\spadesuit}}},\label{eq:RhoPositiv}
\end{equation}
 –which will be our standing assumption from now on–, then $W^{\s,p}\left(Q\right)$
is compactly embedded into $\left(\mathcal{C}(\overline{Q}),\left\Vert \,\cdot\,\right\Vert _{\mathcal{C}(\overline{Q})}\right)$,
with $\left\Vert \,\cdot\,\right\Vert _{\mathcal{C}(\overline{Q})}$
denoting the uniform norm (cf\@. \prettyref{lem: approxSupNorm}).
This allows us to pick a continuous representative of $W^{\s,p}\left(Q\right)$
and gives rise to the compact embeddings 
\[
\iota:W^{\s,p}\left(\Q\right)\hookrightarrow L_{\nu}^{q}\left(\Q\right)\:\text{and }\iota:W_{0}^{\s,p}\left(\Q\right)\hookrightarrow L_{\nu}^{q}\left(\Q\right).
\]
We will also write $L_{\nu}^{q}\coloneqq L_{\nu}^{q}\left(\Q\right)$,
$W^{\s,p}\coloneqq W^{\s,p}\left(\Q\right)$, $W_{0}^{\s,p}\coloneqq W_{0}^{\s,p}\left(\Q\right)$,
and $\left\Vert \,\cdot\,\right\Vert _{L^{\s,p}}=\left\Vert \,\cdot\,\right\Vert _{L^{\s,p}(\Q)}$.
Throughout the paper, we let $\B\in\left\{ \mathscr{B}W_{0}^{\s,p},\mathscr{B}W^{\s,p}\right\} $,
where $\mathscr{B}X\coloneqq\left\{ x\in X:\left\Vert x\right\Vert _{X}\leq1\right\} $
denotes the unit ball in a normed space $\left(X,\left\Vert \,\cdot\,\right\Vert _{X}\right)$.

\subsection{The notion of $n$-widths and approximation order }

Let us now turn to the main object of our investigation. For a normed
vector space $\left(V,\left\Vert \,\cdot\,\right\Vert _{V}\right)$
and a subset $A\subset V$, the \emph{Kolmogorov $n$-widths} of $A$
in $V$, $n\in\N_{0}$, is given by
\begin{align*}
d_{n}^{K}\left(A,V\right) & \coloneqq\inf\left\{ \sup_{x\in A}\inf_{y\in W}\left\Vert x-y\right\Vert _{V}:W<_{n}V\right\} ,
\end{align*}
where $W<_{n}$$V$ means that $W$ is a linear subspace of $V$ of
dimension $\dim W\leq n$. The $n$-width $d_{n}^{K}\left(A,V\right)$
measures the extend to which $A$ can be approximated by $n$-dimensional
subspaces of $V.$ In the same way we consider \emph{Gel{\cprime}fand}
$n$\emph{-widths} of $A$ in $V$, $n\in\N_{0}$, given by
\begin{align*}
d_{n}^{G}\left(A,V\right) & \coloneqq\inf\left\{ \sup_{x\in A\cap U}\left\Vert x\right\Vert _{V}:U<^{n}V\right\} ,
\end{align*}
where $U<^{n}V$ means that $U$ is a linear subspace of $V$ of co-dimension
not exceeding $n$. Recall that $U$ is of co-dimension $n$ if there
exist linearly independent functionals $\phi_{1},\ldots,\phi_{n}$
in the dual space $V'$ of $V$ such that $U=\left\{ v\in V:\phi_{1}\left(v\right)=\cdots=\phi_{n}\left(v\right)=0\right\} $.

Finally, we call 
\begin{align*}
d_{n}^{L}\left(A,V\right) & \coloneqq\inf\left\{ \sup_{x\in A}\left\Vert x-Bx\right\Vert _{V}:B\in\mathscr{L}_{n}(V)\right\} 
\end{align*}
the \emph{linear $n$-width}, where $\mathscr{L}_{n}(V)\coloneqq\left\{ B:V\to V\text{ linear and bounded},B(V)<_{n}V\right\} $
denotes the set of all bounded operators on $V$ of rank at most $n$.
If $A$ is precompact, then both $d_{n}^{K}\left(A,V\right)$ and
$d_{n}^{G}\left(A,V\right)$ converge to zero and we shall investigate
the order of this convergence; to this end, we define
\begin{align*}
\uAO_{\star}\left(A,V\right) & \coloneqq\limsup_{n\to\infty}\frac{\log\left(d_{n}^{\star}\left(A,V\right)\right)}{\log n},\lAO_{\star}\left(A,V\right)\coloneqq\liminf_{n\to\infty}\frac{\log\left(d_{n}^{\star}\left(A,V\right)\right)}{\log n}
\end{align*}
and refer to these quantities as the \emph{upper, resp\@. lower,
Kolmogorov ($\star=K$), }Gel{\cprime}fand\emph{ ($\star=G$), }and
\emph{linear ($\star=L$) approximation order of $A$ in $V$}, respectively.
If the upper and lower approximation order coincide, we simply write
$\AO_{\star}\left(A,V\right)$ for $\star\in\left\{ K,G,L\right\} $.

In this paper we focus on the choice $L_{\nu}^{q}=L_{\nu}^{q}\left(\Q\right)$
for the Banach space $V$, where $\nu$ is a finite Borel measure
with $\supp\left(\nu\right)\subset\mathring{\Q}$ and $\card\left(\text{supp}(\nu)\right)=\infty$
(to exclude trivial cases), and on the unit ball $\B$, which is naturally
embedded into $L_{\nu}^{q}$, in place of the set $A$. We would like
to note that the restriction $\supp\left(\nu\right)\subset\mathring{\Q}$
can be relaxed, but this leads to some technical problems concerning
the lower bounds. This problem could be handled as in \cite{KN2022b}
by considering $\mathcal{D}_{n}^{D}\coloneqq\left\{ Q\in\mathcal{D}_{n}:\partial\Q\cap\overline{Q}=\emptyset\right\} $
instead of $\mathcal{D}_{n}$. For the ease of exposition, we stick
to the assumption imposed on the support of $\nu$.

\subsection{The $L^{q}$-spectrum and partition optimise coarse multifractal
dimension}

\begin{figure}[h]
\center{\begin{tikzpicture}[scale=0.85, every node/.style={transform shape},line cap=round,line join=round,>=triangle 45,x=1cm,y=1cm] \begin{axis}[ x=3.7cm,y=2.3cm, axis lines=middle, axis line style={very thick},ymajorgrids=false, xmajorgrids=false, grid style={thick,densely dotted,black!20}, xlabel= {$t$}, ylabel= {$\beta_\nu (t)$}, xmin=-0.4 , xmax=1.5 , ymin=-0.3, ymax=3.1,x tick style={color=black}, xtick={0, .43,.5,0.6,1},xticklabels = {0,$\phantom{1}\hspace*{-2mm}s_\r$,$\phantom{1}\hspace*{-2mm}a$,$\phantom{1}\hspace*{-2mm}b$,1},  ytick={0,1, 2,3},yticklabels = {0,1, $\overline{\dim}_M(\nu)$ ,3}] \clip(-0.5,-0.3) rectangle (4,4); 
\draw[line width=1pt,smooth,samples=180,domain=-0.3:3.4] plot(\x,{log10(0.001^((\x))+0.3^((\x))+0.1^((\x))+0.599^((\x)))/log10(2)}); 
\draw [line width=01pt,dotted, domain=-0.05 :1.3] plot(\x,{3*(1-\x)});
\draw [line width=01pt,dotted, domain=-0.05 :1.3] plot(\x,{2*(1-\x)});
\draw [line width=01pt,dashed, domain=-0.15 :1.3] plot(\x,{2*\x});
  
\node[circle,draw] (c) at (2.48 ,0 ){\,};

\draw [line width=.7pt,dotted, gray] (0.43 ,0.)--(0.43,0.85); 
\draw [line width=.7pt,dotted, gray] (0.6  ,0 )-- (0.6,1.2);
\draw [line width=.7pt,dotted, gray] (0.5  ,0 )-- (0.5,1);
\end{axis} 
\end{tikzpicture}}

\caption{\label{fig:Moment-generating-function}For $m=3$ the solid line illustrates
the $L^{q}$-spectrum $\beta_{\nu}$ for the self-similar measure
$\nu$ supported on the\emph{ Sierpi\'{n}ski tetraeder} with all
four contraction ratios equal $1/2$ and with probability vector $\left(0.599,0.3,0.001,0.1\right)$;
$\beta_{\nu}\left(0\right)=\overline{\dim}_{M}\left(\nu\right)=2$.
For $\r=2$ (slope of the dashed line) the intersection of the spectrum
and the dashed line determines $s_{\r}$. The dotted lines $t\protect\mapsto3\left(1-t\right)$,
which coincides with the graph of $\beta_{\Lambda}$, and $t\protect\mapsto2\left(1-t\right)$
give rise to the two upper bounds $s_{\r}\protect\leq a\protect\leq b$
with $a\protect\coloneqq\overline{\dim}_{M}\left(\nu\right)/\left(\overline{\dim}_{M}\left(\nu\right)+\r\right)$
and the Birman/Solomjak–Borzov bound \cite{MR0209733,Borzov1971}
$b\protect\coloneqq m/\left(m+\r\right)$.}
\end{figure}
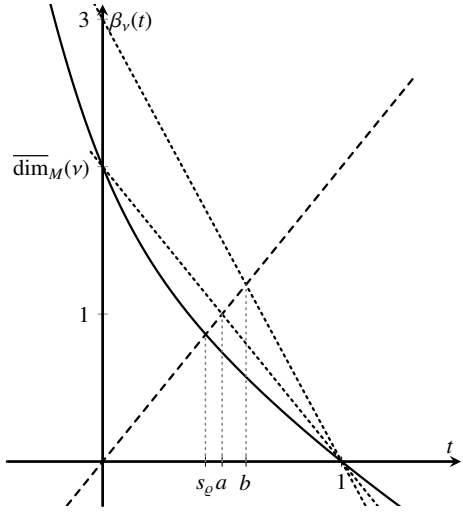

Our main result relies on some previous work, mainly \cite{KN2022,MR4444736}.
The central object of these investigations is the $L^{q}$-spectrum,
which has also attracted some attention in other research areas in
recent years, e.\,g\@. \parencite{MR3919361,KN22b}. For $n\in\N$,
we define the sets of dyadic cubes 
\[
\mathcal{D}_{n}\coloneqq\left\{ Q=\prod_{k=1}^{m}\left(l_{k}2^{-n},(l_{k}+1)2^{-n}\right]:(l_{k})_{k=1}^{m}\in\mathbb{Z}^{m},\nu\left(Q\right)>0\right\} ,\mathcal{D}\coloneqq\bigcup_{n\in\N}\mathcal{D}_{n}
\]
and define the $L^{q}$-spectrum of $\nu$, for $t\in\R_{\geq0}$,
by
\[
\beta_{\nu}(t)\coloneqq\limsup_{n\rightarrow\infty}\beta_{\nu,n}\left(t\right)\:\text{with }\text{\ensuremath{\beta_{\nu,n}\left(t\right)\coloneqq\frac{\log\left(\sum_{Q\in\mathcal{D}_{n}}\nu(Q)^{t}\right)}{\log\left(2^{n}\right)}}}.
\]
Note that $\beta_{\nu}$ is—as a limit superior of convex functions—itself
a convex function and that $\beta_{\nu}(0)$ is equal to the \emph{upper
Minkowski dimension} of $\supp\left(\nu\right)$ denoted by $\overline{\dim}_{M}\left(\nu\right)\coloneqq\overline{\dim}_{M}\left(\supp\left(\nu\right)\right)$
(see \prettyref{fig:Moment-generating-function} and for a proof of
$\beta_{\nu}(0)=\overline{\dim}_{M}\left(\nu\right)$ see e.\,g\@.
\cite{MR1312056}). We also need the\emph{ lower Minkowski dimension}
of $\supp\left(\nu\right)$ denoted by $\underline{\dim}_{M}\left(\nu\right)$,
for which similarly
\begin{align*}
\underline{\dim}_{M}\left(\nu\right) & =\liminf_{n\to\infty}\frac{\log\left(\card\mathcal{D}_{n}\right)}{\log\left(2^{n}\right)}.
\end{align*}
Finally, we set
\[
s_{b}\coloneqq\inf\left\{ t>0:\beta_{\nu}(t)-bt\leq0\right\} \;\text{ for }b>0.
\]
With the help of the concepts of the coarse multifractal formalism
as developed in \cite{KN2022,KN2022b}, necessary conditions for the
existence of the approximate order will be formulated. For this we
employ the concept of the optimise coarse multifractal dimension.
For $q<\infty$ and $\r\coloneqq q\rr=q\left(\s-m/p\right)$, we set

\[
\J\coloneqq\J_{\r}:\mathcal{D}\to\R,Q\mapsto\nu\left(Q\right)\Lambda\left(Q\right)^{\r/m}.
\]
 Note that $\r$ is positive, which is ensured by \prettyref{eq:RhoPositiv}
(this determines the shaded region in \prettyref{fig:A-sketch}, which
is divided into up to 4 zones–referred to as Case I–IV). Then with
\[
N_{\r,n}\left(\a\right)\coloneqq\left\{ Q\in\mathcal{D}_{n}:\J_{\r}\left(Q\right)\geq2^{-\a n}\right\} \:\text{and }\mathcal{N}_{\r,n}\left(\a\right)\coloneqq\card N_{\r,n}\left(\a\right)
\]
we set 
\[
\overline{F}_{\r}\left(\a\right)\coloneqq\limsup_{n\in\N}\frac{\log\mathcal{N}_{\r,n}(\a)}{n\log2}\text{ and }\underline{F}_{\r}\left(\a\right)\coloneqq\liminf_{n\in\N}\frac{\log\mathcal{N}_{\r,n}(\a)}{n\log2}
\]
and define the \emph{upper}, resp\@. \emph{lower},\emph{ optimise
coarse multifractal dimension} by
\begin{equation}
\overline{\mathcal{F}}_{\r}\coloneqq\sup_{\a>0}\frac{\overline{F}_{\r}\left(\a\right)}{\a},\text{resp. }\underline{\mathcal{F}}_{\r}\coloneqq\sup_{\a>0}\frac{\underline{F}_{\r}\left(\a\right)}{\a}.\label{eq:upperlower optimal MFdim}
\end{equation}
At this point we would like to draw your attention to the decisive
result from large deviation theory obtained in \cite{KN2022} in this
context, which states that
\[
s_{\r}=\overline{\mathcal{F}}_{\r}.
\]

\subsection{Main results}

Before summarising our main results, let us first note that by some
general considerations for $\star,\diamond\in\left\{ G,K,L\right\} $
the values of $\uAO_{\star}\left(A,V\right)$ and $\uAO_{\diamond}\left(A,V\right)$
as well as $\lAO_{\star}\left(A,V\right)$ and $\lAO_{\diamond}\left(A,V\right)$
differ by at most $1/2$ (see \prettyref{cor:1st Estimate}). In \cite{MR4444736},
$-1/\left(qs_{\r}\right)$ was obtained as a general upper bound for
the case $1<p\leq q<\infty$. In here, with some more involved methods
concerning both the upper and lower bounds, we are able to provide
the full picture as stated next. For $r\in\left[1,\infty\right]$
we will use the notion $r'\coloneqq r/\left(r-1\right)$ with $\infty'=1$
and $1'=\infty$. Further, we set 
\[
\overline{S}_{\r}\coloneqq\begin{cases}
{\displaystyle \frac{1}{qs_{\r}}}, & \,\text{for }1\leq q<\infty,\\
{\displaystyle \frac{\rr}{\overline{\dim}_{M}\left(\nu\right)}}, & \,\text{for }q=\infty,
\end{cases}\text{and }\quad\underline{S}_{\r}\coloneqq\begin{cases}
{\displaystyle \frac{1}{q\underline{\mathcal{F}}_{\r}}}, & \,\text{for }1\leq q<\infty,\\
{\displaystyle \frac{\phantom{\mathring{J}}\!\!\rr}{\underline{\dim}_{M}\left(\nu\right)}}, & \,\text{for }q=\infty.
\end{cases}
\]

\begin{figure}
\center{\begin{tikzpicture}[scale=.8]     
\begin{axis}[ axis equal image, domain=1:3,xmin=1, xmax=3.2,ymin=1, ymax=3.2,xlabel= {$p$}, ylabel= {$q$},         samples=100,axis y line=center, axis x line=middle, xtick={1.001,1.25,2,3},xticklabels = {1,${m}/{\s}$,2,3},  ytick={1.001, 2,3},yticklabels = { 1, 2,3} ] 
\fill[pattern=%north west lines
dots, pattern color=lightgray
] (1.24,0) rectangle (3.2,3.2);
\draw[line width=.7pt,smooth,dashed,samples=180,domain=1.3:2] plot(\x,{\x/(\x-1)}); 
%\draw [line width=.7pt,dotted, lightgray] (1.25,0)--(1.25,3.18);     
\draw [line width=.7pt,solid, black] (1 ,1)--(3.2,3.2); 
\draw [line width=.7pt,solid, black] (1 ,1)--(1,3.2); 
\draw [line width=.7pt,solid, black] (1 ,1)--(3.5,1); 
\draw [line width=.7pt,solid, black] (1 ,2)--(2,2);  
%\draw [line width=.7pt,dashed, black] (1 ,3)--(3,3); 
%\draw [line width=.7pt,dashed, black] (3 ,1)--(3,3);   
\draw [line width=.7pt,solid, black] (2,2)--(2,3.2);   
\draw (2.15,1.7 ) node[anchor=north west] {\Large{(I)}}; 
\draw (1.1,1.7 ) node[anchor=north west] {\Large{(II)}};
\draw (2.15,2.7 ) node[anchor=north west] {\Large{(III)}};
\draw (1.1,2.7 ) node[anchor=north west] {\Large{(IV.a)}};
\draw (1.5,3.1 ) node[anchor=north west] {\Large{(IV.b)}};
\end{axis} 
\end{tikzpicture}}

\caption{\label{fig:A-sketch}A sketch of the relevant domains for $\left(p,q\right)$:
I: $q\protect\leq p$; II: $p\protect\leq q\protect\leq2$; III: $2\protect\leq p\protect\leq q$;
IV: $p\protect\leq2\protect\leq q$. Our results cover the shaded
region with $p>m/\s$ (equal to $5/4$ in this illustration).}
\end{figure}
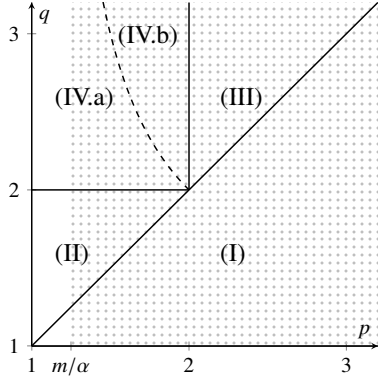

\begin{thm}[Upper approximation order]
\label{thm:MAIN}Assuming \prettyref{eq:RhoPositiv}, we obtain for
the upper approximation order
\begin{align*}
\uAO_{K}\left(\B,L_{\nu}^{q}\right) & =\begin{cases}
-\overline{S}_{\r}\left(q\right)+1/q-1/p, & \text{Cases (I) \& (III): }q\leq p\text{ or }2\leq p\leq q,\\
-\overline{S}_{\r}\left(q\right), & \text{Case (II): }p\leq q\leq2,\\
-\overline{S}_{\r}\left(q\right)+1/q-1/2, & \text{Case (IV): }p\leq2\leq q,
\end{cases}\\
\uAO_{G}\left(\B,L_{\nu}^{q}\right) & =\uAO_{K}\left(B_{\s}^{q'},L_{\nu}^{p'}\right),\\
\uAO_{L}\left(\B,L_{\nu}^{q}\right) & =\max\left\{ \uAO_{G}\left(\B,L_{\nu}^{q}\right),\uAO_{K}\left(\B,L_{\nu}^{q}\right)\right\} .
\end{align*}
\end{thm}

\begin{rem}
\label{rem:Main}Let us make the following remarks:
\begin{enumerate}
\item Note that if $\beta_{\nu}$ is continuous on $\left[0,1\right]$,
then $q\mapsto S_{q\rr}$ is continuous on $\left[1,\infty\right]$.
In fact, we have for all $1\leq p,q\leq\infty$ the closed form expression
\[
1/S_{\r}=\inf\left\{ t>0:\beta_{\nu}\left(t/q\right)-t\rr\leq0\right\} .
\]
This observation leads to the conclusion that the various approximation
orders also dependent continuously on $p$ and $q$ and that the monotonicity
properties of $\beta_{\nu}$ carry over as well.
\item The duality between the Kolmogorov and Gel{\cprime}fand approximation
order is a consequence of the finite-dimensional analogue as stated
in \prettyref{eq:SymmetryCase_G_K} combined with the method of discretisation.
\item If $\nu$ is chosen to be the $s$-dimensional Hausdorff measure restricted
to an $s$-dimensional manifold, then our results are in line with
the orders found in \cite{MR3204029,MR3250051,MR3487230}. This observation
can be naturally generalised to an $s$-Ahlfors–David regular measure
$\nu$. As for the $s$-dimensional Hausdorff measure, the $L^{q}$-spectrum
is then given by the linear function $\beta_{\nu}:t\mapsto\left(1-t\right)s$
(see e.\,g\@. \cite{KN2022b}) and in this situation $s_{\r}$ is
easy to determine.
\item Explicitly, for the upper Gel{\cprime}fand approximation order\emph{
}we have
\[
\uAO_{G}\left(\B,L_{\nu}^{q}\right)=\begin{cases}
-\overline{S}_{\r}\left(q\right)+1/q-1/p, & \text{Cases (I) \& (II): }q\leq p\,\text{or }p\leq q\leq2,\\
-\overline{S}_{\r}\left(q\right), & \text{Case (III): }2\leq p\leq q,\\
-\overline{S}_{\r}\left(q\right)+1/2-1/p, & \text{Case (IV): }p\leq2\leq q
\end{cases}
\]
and, setting $\left(x\right)_{+}\coloneqq\max\left\{ 0,x\right\} $
for $x\in\R$, we have for the upper linear approximation order 
\[
\uAO_{L}\left(\B,L_{\nu}^{q}\right)=\begin{cases}
-\overline{S}_{\r}\left(q\right)+1/q-1/2, & \text{Case (IV.a): }p\leq2\leq q\leq p',\\
-\overline{S}_{\r}\left(q\right)+1/2-1/p, & \text{Case (IV.b): }p\leq2\leq p'<q,\\
-\overline{S}_{\r}\left(q\right)+\left(1/q-1/p\right)_{+}, & \text{else. }
\end{cases}
\]
\end{enumerate}
\end{rem}

\begin{defn}
We define two notions of regularity.
\begin{enumerate}
\item The measure $\nu$ is called \emph{multifractal-regular for $\r<\infty$}
if $\underline{\mathcal{F}}_{\r}=s_{\r}$ and \emph{for} \emph{$\r=\infty$}
if $\underline{\dim}_{M}\left(\nu\right)=\overline{\dim}_{M}\left(\nu\right)$.
\item The measure $\nu$ is called\emph{ $L^{q}$-regular for $\r<\infty$
if}
\begin{enumerate}
\item $\beta_{\nu}\left(t\right)=\liminf_{n}\beta_{\nu,n}\left(t\right)\in\R$
for all $t\in\left(s_{\r}-\varepsilon,s_{\r}\right)\cap\left[0,1\right]$,
for some $\varepsilon>0$, or \label{enu:-LqRegularity1}
\item $\beta_{\nu}\left(s_{\r}\right)=\liminf_{n}\beta_{\nu,n}\left(s_{\r}\right)$
and $\beta_{\nu}$ is differentiable in $s_{\r}$.
\end{enumerate}
\end{enumerate}
\end{defn}

\begin{rem}
We know from \cite[Theorem 1.2]{MR4444736} (see also \cite{KN2022b})
that if $\nu$ is $L^{q}$-regular for $\r<\infty$, then $\nu$ is
multifractal-regular for $\r$, and that the latter property can be
verified in many situations like self-similar measures (without separation
conditions), conformal weak Gibbs measures, condensation systems etc\@.
which are described e.\,g\@. in \cite{KN22b,KN21}. In the case
$q=\infty$, the multifractal-regularity of $\nu$ actually reduces
to the question whether $\overline{\dim}_{M}\left(\nu\right)=\beta_{\nu}\left(0\right)=\liminf_{n}\beta_{\nu,n}\left(0\right)=\underline{\dim}_{M}\left(\nu\right)$,
which is known for many concrete examples.
\end{rem}

\begin{thm}[Lower approximation order]
\label{thm:MAIN_lower} For $\star\in\left\{ K,G,L\right\} $ and
assuming \prettyref{eq:RhoPositiv}, we have the following bounds
on the lower approximation orders:
\[
\overline{S}_{\r}-\underline{S}_{\r}+\uAO_{\star}\left(\B,L_{\nu}^{q}\right)\leq\lAO_{\star}\left(\B,L_{\nu}^{q}\right)\leq\uAO_{\star}\left(\B,L_{\nu}^{q}\right)
\]
and for $q=\infty$ we have the following equalities 
\begin{align*}
\lAO_{K}\left(\B,L_{\nu}^{\infty}\right) & =\begin{cases}
-{\displaystyle \rr/\underline{\dim}_{M}\left(\nu\right)}-1/p, & \text{for }p>2,\\
-{\displaystyle \rr/\underline{\dim}_{M}\left(\nu\right)}-1/2, & \text{for }p\leq2,
\end{cases}\\
\lAO_{G}\left(\B,L_{\nu}^{\infty}\right) & =\lAO_{L}\left(\B,L_{\nu}^{\infty}\right)=\begin{cases}
-{\displaystyle \rr/\underline{\dim}_{M}\left(\nu\right)}, & \text{for }p>2,\\
-{\displaystyle \rr/\underline{\dim}_{M}\left(\nu\right)}+1/2-1/p, & \text{for }p\leq2.
\end{cases}
\end{align*}
In particular, if $\nu$ is multifractal-regular for $\r$, then the
approximation order $\AO_{\star}\left(\B,L_{\nu}^{q}\right)$ exists.
\end{thm}

\begin{rem}
For $q=\infty$, and for measures $\nu$ with $\underline{\dim}_{M}\left(\nu\right)<\overline{\dim}_{M}\left(\nu\right)$
we immediately obtain examples for which $\lAO_{\star}\left(\B,L_{\nu}^{\infty}\right)<\uAO_{\star}\left(\B,L_{\nu}^{\infty}\right)$.
We refer the interested reader to \cite{KN2022}, where a one-dimensional
example of a measure $\nu$ is given which is not $L^{q}$-regular
and for which $\underline{S}_{\r}<\overline{S}_{\r}$ for all $\r>0$.
\end{rem}

\subsection*{Hilbert space cases }

It is known that in the case where the reference space $V$ is a Hilbert
space, the different $n$-widths are closely related \cite[Chapter II: Proposition 5.2 and 8.8]{MR774404}.
In our situation we distinguish the following two cases for $V=L_{\nu}^{2}$:
\begin{itemize}
\item If $p\geq2$, then the following equality holds
\begin{align*}
\uAO_{G}\left(\B,L_{\nu}^{2}\right) & =\uAO_{K}\left(\B,L_{\nu}^{2}\right)=\uAO_{L}\left(\B,L_{\nu}^{2}\right)=\frac{-1}{2s_{\r}}+\frac{1}{2}-\frac{1}{p},
\end{align*}
\item and if $p<2$, then we have the strict inequality
\begin{align*}
\uAO_{G}\left(\B,L_{\nu}^{2}\right) & =\frac{-1}{qs_{\r}}-\frac{1}{p}+\frac{1}{2}<\frac{-1}{2s_{\r}}=\uAO_{K}\left(\B,L_{\nu}^{2}\right)=\uAO_{L}\left(\B,L_{\nu}^{2}\right).
\end{align*}
\end{itemize}
Note that the case where $p=q=2$ has been treated in \cite{MR4444736}
in the context of polyharmonic operators.\emph{ }

\subsection*{Fractal-geometric upper bounds \label{subsec:Geometric-upper-Bounds}}

Using the convexity of $\beta_{\nu}$ we easily deduce the following
upper bounds (which correspond to the values $a$ and $b$ as illustrated
in \prettyref{fig:Moment-generating-function})
\begin{equation}
-\overline{S}_{\r}\leq-\frac{\rr}{\overline{\dim}_{M}\left(\nu\right)}-\frac{1}{q}\leq-\frac{\rr}{m}-\frac{1}{q}.\label{eq:geometric/general_upper_bound}
\end{equation}
The weakest bound on the right hand side was obtained in \cite{MR0209733,Borzov1971}
and the improved version in the middle is mentioned in \cite{MR4444736}.
This inequality can be used for each of the cases in \prettyref{thm:MAIN}.
Moreover, note that in the first inequality, equality actually holds
for $q=\infty$, while a strict inequality holds for $q<\infty$ and
for $\beta_{\nu}$, which is strictly convex.
\begin{rem}
\label{rem:Note_Upper_bound for q<p} Note that the upper bound $-\overline{S}_{\r}-1/p+1/q$
of the Kolmogorov approximation order for the case $q\leq p$ is necessarily
negative since by \prettyref{eq:geometric/general_upper_bound} we
obtain in this case $-\s/m<0$ as an upper bound.
\end{rem}

\subsection{Outline}

The paper is organised as follows. In \prettyref{sec:Preliminaries}
we give the basic functional analytic background from approximation
theory. In particular, we state the frequently used inequalities and
properties of the widths for finite-dimensional vector spaces. \prettyref{sec:Poincar=0000E9-inequalities+upper bounds}
is devoted to the upper bounds of our main results. For this, the
ideas of discretisation and optimal partitions are suitably merged
(cf\@. Propositions \ref{prop:UpperBound_Discretisation} and \ref{prop:ImprovedupperBoundCase}
as our main results in this respect). Finally, \prettyref{sec:Coarse-Muiltifractal-formalism}
borrows ideas from the coarse multifractal formalism to establish
the lower bounds under consideration. Again, an interplay between
discretisation methods for the approximation order and finding optimal
disjoint families of cubes obtained with large deviation techniques
is employed (cf\@. Propositions \ref{prop:discretesation_lower bound widths}
and \ref{prop:AllLowerBounds}). In all the subsequent sections, \prettyref{eq:RhoPositiv}
is assumed throughout. Note that we will give the proofs of the upper
bounds for $\B=\mathscr{B}W^{\s,p}$, while the proofs of the lower
bounds are stated for $\B=\BB$. The claims then follow by the trivial
observation 
\begin{equation}
\begin{split}\uAO_{\star}\left(\BB,L_{\nu}^{q}\right) & \leq\uAO_{\star}\left(\mathscr{B}W^{\s,p},L_{\nu}^{q}\right),\\
\lAO_{\star}\left(\BB,L_{\nu}^{q}\right) & \leq\lAO_{\star}\left(\mathscr{B}W^{\s,p},L_{\nu}^{q}\right).
\end{split}
\label{eq:trivialObservation}
\end{equation}

\section{Preliminaries \label{sec:Preliminaries}}

To provide the prerequisites for our approach, in this section we
give some basic facts from approximation theory, mainly for the finite-dimensional
case, following the textbooks \cite{MR1393437,MR774404}.

For $T\in\mathscr{L}\left(X,V\right)$ a linear operator and for the
unit ball $\mathscr{B}X$ in $X$ we set $d_{n}^{\star}\left(T\right)\coloneqq d_{n}^{\star}\left(T\left(\mathscr{B}X\right),V\right)$,
$\star\in\left\{ K,G,L\right\} $. If $T$ denotes a natural injection
$\iota:\mathscr{B}X\hookrightarrow V$, e.\,g\@. of a Sobolev spaces
$X$, then we write $d_{n}^{\star}\left(\mathscr{B}X,V\right)$ instead
of $d_{n}^{\star}\left(\iota\right)$. We will make the following
general observations.
\begin{lem}[{\cite[Chapter II: Proposition 3.2.]{MR774404}}]
\label{lem:-Gelfand_extension} Assume that X and Y are normed linear
spaces, and X is a subspace of Y, then 
\[
d_{n}^{G}(A,X)=d_{n}^{G}(A,Y).
\]
\end{lem}

The following result can be found in \cite[Chapter II: p.  21, Proposition 5.1, Corollar 8.5, Theorem 8.9]{MR774404}.
\begin{lem}
\label{lem:GeneralBounds dn,dn} We have 
\[
\max\left\{ d_{n}^{K}\left(T\right),d_{n}^{G}\left(T\right)\right\} \leq d_{n}^{L}\left(T\right)\leq\left(\sqrt{n}+1\right)\min\left\{ d_{n}^{K}\left(T\right),d_{n}^{G}\left(T\right)\right\} .
\]
 
\end{lem}

If $T$ is a compact operator, then as a consequence we obtain that
the various approximation orders can only differ by $1/2$.
\begin{cor}
\label{cor:1st Estimate} For $\star,\diamond\in\left\{ G,K,L\right\} $
such that the associated approximation orders are finite, we have
\[
\left|\uAO_{\star}\left(T\left(\mathscr{B}X\right),V\right)-\uAO_{\diamond}\left(T\left(\mathscr{B}X\right),V\right)\right|\leq1/2
\]
and 
\[
\left|\lAO_{\star}\left(T\left(\mathscr{B}X\right),V\right)-\lAO_{\diamond}\left(T\left(\mathscr{B}X\right),V\right)\right|\leq1/2.
\]
\end{cor}

\begin{proof}
This follows for $\star\in\left\{ G,K,L\right\} $ and $\diamond\in\left\{ G,K\right\} $
since 
\[
0\leq\left|\frac{\text{\ensuremath{\log d_{n}^{\star}\left(T\right)}}}{\log n}-\frac{\log d_{n}^{\diamond}\left(T\right)}{\log n}\right|\leq\frac{1}{\log n}\log\frac{d_{n}^{L}\left(T\right)}{\min\left\{ d_{n}^{K}\left(T\right),d_{n}^{G}\left(T\right)\right\} }\leq\frac{\log\left(\sqrt{n}+1\right)}{\log n}.
\]
\end{proof}
\begin{notation}
For two functions $f,g$ depending on $x>0$ we write $f\ll g$ if
there exists a constant $C>0$ such that $0<f(x)\leq Cg(x)$ for all
$x$ sufficiently large. If $f\ll g$ and $g\ll f$ holds, we write
$f\asymp g$. 
\end{notation}

We will need some further facts about $n$-widths for the finite-dimensional
case. The following results can be found e.\,g\@. in \cite[pp.  409, 411, 482]{MR1393437}
and are valid for $p,q\in\left[1,\infty\right]$. Again, $p'$ denotes
the dual value of $p\in\left[1,\infty\right]$. We let $\ell_{p}^{M}$
denote the Banach space $\R^{M}$ with respect to the $p$-norm $\left\Vert x\right\Vert _{\ell_{p}^{M}}\coloneqq\left(\sum_{i=1}^{M}\left|x_{i}\right|^{p}\right)^{1/p}$,
for $1\leq p<\infty$ and $\left\Vert x\right\Vert _{\ell_{\infty}^{M}}\coloneqq\max_{i=1,\ldots,M}\left|x_{i}\right|$,
for $p=\infty$, $x\in\R^{M}$, and for the unit ball in $\ell_{p}^{M}$
we write $\b_{p}^{M}\coloneqq\mathscr{B}\ell_{p}^{M}$.
\begin{lem}
\label{lem:dnfinteDim} For $1\leq n\leq k$, we have 
\[
d_{n}^{L}\left(\b_{p}^{k},\ell_{q}^{k}\right)=d_{n}^{L}\left(\b_{q'}^{k},\ell_{p'}^{k}\right)\asymp\begin{cases}
d_{n}^{K}\left(\left(\b_{p}^{k},\ell_{q}^{k}\right)\right), & \text{for }q\leq p',\\
d_{n}^{K}\left(\left(\b_{q'}^{k},\ell_{p'}^{k}\right)\right), & \text{for }p'\leq q,
\end{cases}
\]
and 
\begin{equation}
d_{n}^{G}\left(\b_{p}^{k},\ell_{q}^{k}\right)=d_{n}^{K}\left(\b_{q'}^{k},\ell_{p'}^{k}\right).\label{eq:SymmetryCase_G_K}
\end{equation}
\end{lem}

The following lemma is a consequence of a more general result in \cite{MR687610}. 

\begin{lem}[{\cites[p. 411]{MR1393437}[Theorem 1.3]{Kas77}}]
\label{lem:KolmogorovFiniteDim} For the Kolmogorov widths we have
\[
d_{n}^{K}\left(\left(\b_{p}^{2n},\ell_{q}^{2n}\right)\right)\asymp\begin{cases}
n^{1/q-1/p}, & \text{for }q\leq p\text{ or }2\leq p\leq q,\\
1, & \text{for }p\leq q\leq2,\\
n^{1/q-1/2}, & \text{for }p\leq2\leq q.
\end{cases}
\]
\end{lem}

We deduce a similar result for the Gel{\cprime}fand widths by employing
the duality.
\begin{cor}
\label{cor:GelfandFiniteDim} For the Gel{\cprime}fand widths we
have
\[
d_{n}^{G}\left(\left(\b_{p}^{2n},\ell_{q}^{2n}\right)\right)\asymp\begin{cases}
n^{1/q-1/p}, & \text{for }q\leq p\text{ or }p\leq q\leq2,\\
1, & \text{for }2\leq p<q,\\
n^{1/2-1/p}, & \text{for }p\leq2\leq q.
\end{cases}
\]
\end{cor}

\begin{proof}
This follows by combining the second part of \prettyref{lem:dnfinteDim}
with \prettyref{lem:KolmogorovFiniteDim} by observing $1/p'-1/q'=1/q-1/p$
and $1/2-1/p'=1/q-1/2$.
\end{proof}
For the linear widths we have the following asymptotic.
\begin{cor}
\label{cor:LinearFiniteDim} For the linear widths we have
\[
d_{n}^{L}\left(\left(\b_{p}^{2n},\ell_{q}^{2n}\right)\right)\asymp\begin{cases}
n^{1/q-1/p}, & \text{for }q\leq p,\\
1, & \text{for }2\leq p\leq q\text{ or }p\leq q\leq2,\\
n^{1/q-1/2}, & \text{for }p\leq2\leq q\leq p',\\
n^{1/2-1/p}, & \text{for }p\leq2,p'\leq q.
\end{cases}
\]
\end{cor}

\begin{proof}
This follows immediately from the first part of \prettyref{lem:dnfinteDim},
\prettyref{lem:KolmogorovFiniteDim}, and the fact that 
\[
\max\left\{ 1/q-1/2,1/2-1/p\right\} =\begin{cases}
1/q-1/2, & \text{for }q\leq p',\\
1/2-1/p, & \text{for }q\geq p'.
\end{cases}
\]
\end{proof}

\section{\label{sec:Poincar=0000E9-inequalities+upper bounds}Poincaré inequalities
and upper bounds }

The proof of the upper bounds combines ideas from the proof of Ka\v sin's
Theorem as provided in \cite[Theorem 5.4 in Chapter 14]{MR1393437}
with some recent work on optimal partitions elaborated in the context
of Kre\u{\i}n–Feller operators in \cite{KN2022,KN2022b}. To this
end, for $q<\infty$, we consider the finite partition of dyadic cubes
of $\Q$ given by
\begin{equation}
P_{t}\coloneqq\left\{ Q\in\mathcal{D}:\J\left(Q\right)<t\,\,\,\&\,\,\,\exists Q'\in\mathcal{D}_{\left|\log_{2}\Lambda\left(Q\right)\right|/m-1}:Q'\supset Q\,\,\,\&\,\,\,\J(Q')\geq t\right\} \label{eq:partition}
\end{equation}
for $t>0$, and we recall from \cite{KN2022,KN2022b} that 
\begin{equation}
\limsup_{t\downarrow0}\frac{\log\left(\card\left(P_{t}\right)\right)}{-\log(t)}=s_{\r}.\label{eq:partition_entropy}
\end{equation}
Further, we make use of some ideas and results from \parencite[§ 3]{MR0217487}
and \parencite{MR0482138}. For a half-open cube $Q\in\mathcal{D}$
and for every $u\in W^{\s,p}\left(Q\right)$, we associate a polynomial
$r\in\R\left[x_{1},\ldots,x_{m}\right]$ of degree at most $\s-1$\textbf{
}satisfying the conditions
\begin{equation}
\int_{Q}x^{k}r(x)\d\Lambda(x)=\int_{Q}x^{k}u(x)\d\Lambda(x)\,\,\text{for all}\,\,|k|\leq\s-1.\label{eq:Polynomial}
\end{equation}
By an application of \emph{Hilbert's Projection Theorem} with respect
to $L^{2}(Q)$, we have that $r$ is uniquely determined by \prettyref{eq:Polynomial}
and set $\PP_{Q}u\coloneqq r$. Note that $\PP_{Q}$ defines a linear
projection operator which maps $W^{\s,p}\left(Q\right)$ to the finite-dimensional
space of polynomials in $m$ variables of degree not exceeding $\s-1$.
For this space we fix the basis $\mathcal{B}_{\s}\coloneqq\left\{ p_{1},\ldots,p_{\kappa}\right\} $
for some $\kappa\in\N$. We consider sequences $t_{k}\searrow0$ and
a strictly increasing sequence $\left(j_{k}\right)\in\N^{\N}$ such
that
\begin{gather}
\begin{split} & \sup\left\{ \left|\log_{2}\Lambda\left(Q\right)-\log_{2}\Lambda(\tilde{Q})\right|:P_{t_{n+1}}\in Q\subset\tilde{Q}\ni P_{t_{n}},n\in\N\right\} <\infty,\\
 & \sup_{k\in\N}\left(j_{k+1}-j_{k}\right)<\infty.
\end{split}
\label{eq:boundOnCardinalityOFCubes}
\end{gather}
The first bound can be ensured, for example, by $\eta\leq t_{n+1}/t_{n}$,
$n\in\N,$ for some $\eta\in\left(0,1\right)$. Let us define for
$k\in\N$
\[
\boldsymbol{P}_{k}\coloneqq\begin{cases}
P_{t_{k}}, & \text{for }q<\infty,\\
\mathcal{D}_{n_{k}}, & \text{for }q=\infty,
\end{cases}
\]
the partition of $\Q$ into half-open cubes defined in \prettyref{eq:partition}
for the first case and the partition of equal sized dyadic cubes neglecting
cubes with $\nu$ measure zero in the second case. Its cardinality
is denoted by $M_{k}\coloneqq\card\left(\boldsymbol{P}_{k}\right)$.

Fix $Q\in\mathcal{D}_{n}$ with side length $c_{Q}=\left(\Lambda\left(Q\right)\right)^{1/m}=2^{-n}$
for some $n\in\N$ and let $b\in\R^{m}$ denote the midpoint of $Q$.
Then for 
\begin{equation}
\phi_{Q}:\R^{m}\to\R^{m},x\mapsto c_{Q}x+b\label{eq:Def_phi_Q}
\end{equation}
we have $\phi_{Q}\left(\Q\right)=Q$. Then $\mathcal{B}_{k}\coloneqq\left\{ p_{r,Q}\coloneqq p_{r}\circ\phi_{Q}^{-1}:r=1,\ldots,\kappa,Q\in\boldsymbol{P}_{k}\right\} $
is a basis of the the subspace $\mathcal{P}_{k}\coloneqq\mathcal{P}\left(\boldsymbol{P}_{k},\s-1\right)$
of piecewise-polynomial functions which restrict on each cube $Q\in\boldsymbol{P}_{k}$
to a polynomial of degree at most $\s-1$ for each partition $\boldsymbol{P}_{k}\subset\mathcal{D}$
and with $\card\left(\mathcal{B}_{k}\right)=\kappa M_{k}$. Let us
denote the coordinate mapping by 
\begin{align*}
I_{k} & :\R^{\kappa M_{k}}\to\mathcal{P}_{k},\:\left(c_{k,Q}\right)\mapsto\sum_{Q\in\boldsymbol{P}_{k}}\sum_{r=1}^{\kappa}c_{r,Q}\widetilde{p}_{r,Q}\;\text{with}\,\,\widetilde{p}_{r,Q}\coloneqq\Lambda\left(Q\right)^{\rr/m}p_{r,Q}.
\end{align*}

Since all norms on a finite-dimensional vector space are equivalent
and by the definition of $p_{r,Q}$ we have for all $1\leq p,q\leq\infty$
uniformly in $b=\left(b_{r}\right)\in\R^{\kappa}$ and $Q\in\mathcal{D}$,
\begin{equation}
\left\Vert b\right\Vert _{\ell_{q}^{\kappa}}\asymp\left\Vert b\right\Vert _{\ell_{p}^{\kappa}}\asymp\left\Vert \sum_{r=1}^{\kappa}b_{r}p_{r}\right\Vert _{\mathcal{C}\left(\Q\right)}=\left\Vert \sum_{r=1}^{\kappa}b_{r}p_{r,Q}\right\Vert _{\mathcal{C}(Q)}.\label{eq:equivalentNorm}
\end{equation}
Next, define 
\begin{align*}
U_{k}:W^{\s,p}\left(\Q\right) & \rightarrow\mathcal{P}_{k},\;f\mapsto\sum_{Q\in\boldsymbol{P}_{k}}\1_{Q}\PP_{Q}f,
\end{align*}
where $\1_{Q}$ denotes the characteristic function on the cube $Q$
and set
\[
V_{1}\coloneqq U_{1},V_{k}\coloneqq U_{k}-U_{k-1},\;k\geq2.
\]
The following Poincaré/Wirtinger/Sobolev type inequality will be crucial
for the upper bounds derived in this section. For a reference we refer
to \parencite[Lemma 3.1]{MR0217487}.
\begin{lem}
\label{lem: approxSupNorm} For all cubes $Q\in\mathcal{D}$ and all
$u\in W^{\s,p}\left(Q\right)$ we have 
\[
\left\Vert u-\PP_{Q}u\right\Vert _{\mathcal{C}(\overline{Q})}\ll\Lambda(Q)^{\rr/m}\left\Vert u\right\Vert _{L^{\s,p}(Q)}.
\]
\end{lem}

For each $Q\in P_{n}$ there exists a unique element $\tilde{Q}\in P_{n-1}$
such that $Q\subset\tilde{Q}$ and for this element we write $\tilde{Q}_{Q}$.
This notation will be used in the proof of the following lemma. 
\begin{lem}
\label{lem:BasicEstimateV_kf} For $f\in W^{\s,p}$, $n\in\N$, $Q\in\mathcal{D}_{n-1}$
and $V_{n}f=\sum_{Q\in\boldsymbol{P}_{k}}\sum_{r=1}^{\kappa}c_{r,Q}p_{r,Q}$,
we have
\begin{enumerate}
\item $\left\Vert V_{n}f\right\Vert _{\mathcal{C}\left(Q\right)}\ll\Lambda\left(Q\right)^{\rr/m}\left\Vert f\right\Vert _{L^{\s,p}\left(Q\right)}$,
\item $f=\sum_{k\in\N}V_{k}f$ as an element of $L_{\nu}^{q}$, 
\item $\left\Vert \left(c_{r,Q}\right)\right\Vert _{\ell_{p}^{\kappa M_{k}}}\ll\left\Vert f\right\Vert _{L^{\s,p}}.$
\end{enumerate}
\end{lem}

\begin{proof}
\emph{ad (1):} With $f\in W^{\s,p}$ and $Q\in\mathcal{D}_{n-1}$
we have by \prettyref{lem: approxSupNorm} 
\begin{align*}
\left\Vert V_{n}f\right\Vert _{\mathcal{C}\left(Q\right)} & \leq\left\Vert U_{n}f-f\right\Vert _{\mathcal{C}\left(Q\right)}+\left\Vert f-U_{n-1}f\right\Vert _{\mathcal{C}\left(Q\right)}\\
 & =\max_{P_{n}\ni Q'\subset Q}\left\Vert P_{Q'}f-f\right\Vert _{\mathcal{C}\left(Q'\right)}+\left\Vert f-\PP_{Q}f\right\Vert _{\mathcal{C}\left(Q\right)}\\
 & \ll\Lambda\left(Q\right)^{\rr/m}\left(\max_{P_{n}\ni Q'\subset Q}\left\Vert f\right\Vert _{L^{\s,p}\left(Q'\right)}+\left\Vert f\right\Vert _{L^{\s,p}\left(Q\right)}\right)\\
 & \leq\begin{cases}
\Lambda\left(Q\right)^{\rr/m}\left(\left\Vert f\right\Vert _{L^{\s,\infty}\left(Q\right)}+\left\Vert f\right\Vert _{L^{\s,\infty}\left(Q\right)}\right), & \text{for }p=\infty,\\
\Lambda\left(Q\right)^{\rr/m}\left(\left(\sum_{P_{n}\ni Q'\subset Q}\left\Vert f\right\Vert _{L^{\s,p}\left(Q'\right)}^{p}\right)^{1/p}+\left\Vert f\right\Vert _{L^{\s,p}\left(Q\right)}\right), & \text{for }p<\infty
\end{cases}\\
 & \ll\Lambda\left(Q\right)^{\rr/m}\left\Vert f\right\Vert _{L^{\s,p}\left(Q\right)}
\end{align*}
where for $p<\infty$ we used \prettyref{eq:boundOnCardinalityOFCubes}.
This shows the first claim. 

\emph{ad (2):} We first consider $q<\infty$. Using \prettyref{lem: approxSupNorm}
and the upper bound discussed in \prettyref{rem:Note_Upper_bound for q<p},
we find
\begin{align*}
\lim_{n\to\infty}\left\Vert f-\sum_{k=1}^{n}V_{k}f\right\Vert _{L_{\nu}^{q}} & =\lim_{n\to\infty}\left(\int_{\Q}\left\vert f-U_{n}f\right\vert ^{q}\d\nu\right)^{1/q}\leq\lim_{n\to\infty}\left(\sum_{Q\in P_{n}}\nu\left(Q\right)\left\Vert f-\PP_{Q}f\right\Vert _{\mathcal{C}\left(Q\right)}^{q}\right)^{1/q}\\
 & \ll\lim_{n\to\infty}\left(\max_{Q\in P_{n}}\J_{\r}\left(Q\right)\right)^{1/q}\left(\sum_{Q\in P_{n}}\left\Vert f\right\Vert _{L^{\s,p}\left(Q\right)}^{q}\right)^{1/q}\\
 & \leq\lim_{n\to\infty}\left(\max_{Q\in P_{n}}\J_{\r}\left(Q\right)\right)^{1/q}M_{n}^{\left(1/q-1/p\right)_{+}}\left\Vert f\right\Vert _{L^{\s,p}\left(Q\right)}=0.
\end{align*}
Note that for $x\in\R^{M}$ and $p\leq q$, we used the well-known
chain of inequalities 
\begin{equation}
\left\Vert x\right\Vert _{\ell_{q}^{M}}\leq\left\Vert x\right\Vert _{\ell_{p}^{M}}\leq M^{1/p-1/q}\left\Vert x\right\Vert _{\ell_{p}^{M}}.\label{eq:H=0000F6lderFiniteDim}
\end{equation}
Similarly, for $q=\infty$, we get 
\begin{align*}
\lim_{n\to\infty}\left\Vert f-\sum_{k=1}^{n}V_{k}f\right\Vert _{L_{\nu}^{\infty}} & =\lim_{n\to\infty}\left\Vert f-U_{n}f\right\Vert _{L_{\nu}^{\infty}}\leq\lim_{n\to\infty}\max_{Q\in P_{n}}\left\Vert f-\PP_{Q}f\right\Vert _{\mathcal{C}\left(Q\right)}\\
 & \ll\lim_{n\to\infty}\max_{Q\in P_{n}}\Lambda\left(Q\right)^{\rr/m}\left\Vert f\right\Vert _{L^{\s,p}\left(Q\right)}\\
 & \leq\begin{cases}
\lim_{n\to\infty}2^{-n\rr}\left\Vert f\right\Vert _{L^{\s,\infty}}, & \text{for }p=\infty,\\
\lim_{n\to\infty}2^{-n\rr}\left(\sum_{Q\in\mathcal{D}_{n}}\left\Vert f\right\Vert _{L^{\s,p}\left(Q\right)}^{p}\right)^{1/p}, & \text{for }p<\infty
\end{cases}\\
 & =\lim_{n\to\infty}2^{-n\rr}\left\Vert f\right\Vert _{L^{\s,p}}=0.
\end{align*}
 This proves the second claim. 

\emph{ad (3):} Using \prettyref{eq:boundOnCardinalityOFCubes} and
\prettyref{eq:equivalentNorm}, we have 
\begin{align*}
\left\Vert c\right\Vert _{\ell_{p}^{\kappa M_{k}}} & =\left(\sum_{Q\in\boldsymbol{P}_{k}}\left(\left(\sum_{r=1}^{\kappa}\left|c_{r,Q}\right|^{p}\right)^{1/p}\right)^{p}\right)\ll\left(\sum_{Q\in\boldsymbol{P}_{k}}\left\Vert \sum_{r=1}^{\kappa}c_{r,Q}p_{r,Q}\right\Vert _{\mathcal{C}(Q)}^{p}\right)^{1/p}\\
 & \ll\left(\sum_{Q\in\boldsymbol{P}_{k}}\left\Vert \Lambda\left(Q\right)^{-\s/m-1/p}\sum_{r=1}^{\kappa}c_{r,Q}\tilde{p}_{r,Q}\right\Vert _{\mathcal{C}\left(Q\right)}^{p}\right)^{1/p}\\
 & =\left(\sum_{Q\in\boldsymbol{P}_{k}}\left(\Lambda\left(Q\right)^{-\s/m-1/p}\left\Vert V_{k}f\right\Vert _{\mathcal{C}\left(\tilde{Q}_{Q}\right)}\right)^{p}\right)^{1/p}\ll\left(\sum_{Q\in\boldsymbol{P}_{k}}\left\Vert f\right\Vert _{L^{\s,p}\left(\tilde{Q}_{Q}\right)}^{p}\right)^{1/p}\ll\left\Vert f\right\Vert _{L^{\s,p}}.
\end{align*}
\end{proof}
For the following key observation we use the shorthand notation
\[
T_{k}\coloneqq\begin{cases}
t_{k}^{1/q}, & \text{for }q<\infty,\\
2^{-k\rr}, & \text{for }q=\infty.
\end{cases}
\]

\begin{prop}[Discretisation – upper bound]
\label{prop:UpperBound_Discretisation} For $\star\in\left\{ K,G,L\right\} $,
assuming \prettyref{eq:boundOnCardinalityOFCubes} and $\left(n_{k}\right)\in\N_{0}^{\N}$
with $\sum_{k\in\N}n_{k}=n\in\N$, we have 
\[
d_{n}^{\star}\left(\mathscr{B}W^{\s,p},L_{\nu}^{q}\right)\ll\sum_{k}T_{k}d_{n_{k}}^{\star}\left(\b_{p}^{\kappa M_{k}},\ell_{q}^{\kappa M_{k}}\right).
\]
\end{prop}

\begin{proof}
First, we consider $\star=L$. By the definition of $d_{k}^{L}\coloneqq d_{n_{k}}^{L}\left(\b_{p}^{\kappa M_{k}},\ell_{q}^{\kappa M_{k}}\right)$,
for each $k\in\N$, there exists a linear operator $R_{k}\in\mathscr{L}\left(\ell_{q}^{\kappa M_{k}},n_{k}\right)$
on $\ell_{q}^{^{\kappa M_{k}}}$ of rank at most $n_{k}$ with
\[
\left\Vert a-R_{k}a\right\Vert _{\ell_{q}^{\kappa M_{k}}}\leq d_{k}^{L}\left\Vert a\right\Vert _{\ell_{p}^{\kappa M_{k}}}
\]
for all $a\in\ell_{p}^{M_{k}}$ and uniformly in $k$. Hence, with
$\mathcal{R}_{k}\coloneqq I_{k}R_{k}I_{k}^{-1}\in\mathscr{L}_{n_{k}}\left(\mathcal{P}_{k}\right)$,
we find with \prettyref{lem:BasicEstimateV_kf} for $\varphi\in W^{\s,p}$
and $\left(c_{r,Q}\right)\in\ell_{p}^{\kappa M_{k}}$ such that $V_{k}\varphi=\sum_{Q}\sum_{r}c_{r,Q}\tilde{p}_{r,Q}\in\mathcal{P}_{k}$,
and for $q<\infty$,

\begin{align*}
 & \!\!\!\!\left\Vert V_{k}\varphi-\mathcal{R}_{k}V_{k}\varphi\right\Vert _{L_{\nu}^{q}}^{q}\\
 & =\left\Vert \left(\text{Id}-I_{k}R_{k}I_{k}^{-1}\right)\sum_{Q\in\boldsymbol{P}_{k}}\sum_{r=1}^{\kappa}c_{r,Q}\tilde{p}_{r,Q}\right\Vert _{L_{\nu}^{q}}^{q}=\left\Vert \left(\sum_{Q\in\boldsymbol{P}_{k}}\sum_{r=1}^{\kappa}\left(c_{r,Q}-\left(R_{k}c\right)_{r,Q}\right)\tilde{p}_{r,Q}\right)\right\Vert _{L_{\nu}^{q}}^{q}\\
 & =\sum_{Q\in\boldsymbol{P}_{k}}\int_{Q}\left|\sum_{r=1}^{\kappa}\left(c_{r,Q}-\left(R_{k}c\right)_{r,Q}\right)\tilde{p}_{r,Q}\right|^{q}\d\nu\leq\sum_{Q\in\boldsymbol{P}_{k}}\nu\left(Q\right)\left\Vert \sum_{r=1}^{\kappa}\left(c_{r,Q}-\left(R_{k}c\right)_{r,Q}\right)\tilde{p}_{r,Q}\right\Vert _{\mathcal{C}(Q)}^{q}\\
 & =\sum_{Q\in\boldsymbol{P}_{k}}\nu\left(Q\right)\Lambda\left(Q\right)^{\r/m}\left\Vert \sum_{r=1}^{\kappa}\left(c_{r,Q}-\left(R_{k}c\right)_{r,Q}\right)p_{r,Q}\right\Vert _{\mathcal{C}(Q)}^{q}\\
 & \leq\max_{Q\in\boldsymbol{P}_{k}}\J_{\r}\left(Q\right)\sum_{Q\in\boldsymbol{P}_{k}}\left\Vert \sum_{r=1}^{\kappa}\left(c_{r,Q}-\left(R_{k}c\right)_{r,Q}\right)p_{r,Q}\right\Vert _{\mathcal{C}(Q)}^{q}\\
 & \ll t_{k}\sum_{Q\in\boldsymbol{P}_{k}}\left\Vert \left(c_{r,Q}-\left(R_{k}c\right)_{r,Q}\right)_{r=1}^{\kappa}\right\Vert _{\ell_{q}^{\kappa}}^{q}\leq t_{k}\left(d_{k}^{L}\left\Vert c\right\Vert _{\ell_{p}^{\kappa M_{k}}}\right)^{q}\ll t_{k}\left(d_{k}^{L}\left\Vert \varphi\right\Vert _{L^{\s,p}}\right)^{q}.
\end{align*}
For $q=\infty$ we get
\begin{align*}
 & \!\!\!\!\left\Vert V_{k}\varphi-\mathcal{R}_{k}V_{k}\varphi\right\Vert _{L_{\nu}^{\infty}}\\
 & =\left\Vert \sum_{Q\in\boldsymbol{P}_{k}}\sum_{r=1}^{\kappa}c_{r,Q}\tilde{p}_{r,Q}-I_{k}R_{k}I_{k}^{-1}\sum_{Q\in\boldsymbol{P}_{k}}\sum_{r=1}^{\kappa}c_{r,Q}\tilde{p}_{r,Q}\right\Vert _{L_{\nu}^{\infty}}=\left\Vert \left(\sum_{Q\in\boldsymbol{P}_{k}}\sum_{r=1}^{\kappa}\left(c_{r,Q}-\left(R_{k}c\right)_{r,Q}\right)\tilde{p}_{r,Q}\right)\right\Vert _{L_{\nu}^{\infty}}\\
 & \leq\max_{Q\in\boldsymbol{P}_{k}}\left\Vert \sum_{r=1}^{\kappa}\left(c_{r,Q}-\left(R_{k}c\right)_{r,Q}\right)\tilde{p}_{r,Q}\right\Vert _{\mathcal{C}(Q)}\leq\max_{Q\in\boldsymbol{P}_{k}}\Lambda\left(Q\right)^{\rr/m}\left\Vert \sum_{r=1}^{\kappa}\left(c_{r,Q}-\left(R_{k}c\right)_{r,Q}\right)p_{r,Q}\right\Vert _{\mathcal{C}(Q)}\\
 & =2^{-k\rr}\max_{Q\in\boldsymbol{P}_{k}}\left\Vert \sum_{r=1}^{\kappa}\left(c_{r,Q}-\left(R_{k}c\right)_{r,Q}\right)p_{r,Q}\right\Vert _{\mathcal{C}(Q)}\leq2^{-k\rr}\max_{Q\in\boldsymbol{P}_{k}}\left\Vert \left(c_{r,Q}-\left(R_{k}c\right)_{r,Q}\right)_{r=1}^{\kappa}\right\Vert _{\ell_{\infty}^{\kappa}}\\
 & =2^{-k\rr}\left\Vert c-R_{k}c\right\Vert _{\ell_{\infty}^{\kappa M_{k}}}\leq2^{-k\r}d_{k}^{L}\left\Vert c\right\Vert _{\ell_{p}^{\kappa M_{k}}}\ll2^{-k\r}d_{k}^{L}\left\Vert \varphi\right\Vert _{L^{\s,p}}
\end{align*}
using \prettyref{lem:BasicEstimateV_kf} and \prettyref{eq:equivalentNorm}.
Now we define the linear operator $\mathcal{Q}_{n}\coloneqq\sum_{k\in\N}\mathcal{R}_{k}V_{k}$,
which is of rank at most $n=\sum n_{k}$ and conclude for $q<\infty$
\begin{align*}
d_{n}^{L}\left(\mathscr{B}W^{\s,p},L_{\nu}^{q}\right) & =\inf_{B\in\mathscr{L}_{n}\left(L_{\nu}^{q}\right)}\sup_{\varphi\in\B}\left\Vert \varphi-B\varphi\right\Vert _{L_{\nu}^{q}}\leq\sup_{\varphi\in\B}\left\Vert \varphi-\mathcal{Q}_{n}\varphi\right\Vert _{L_{\nu}^{q}}\\
 & =\sup_{\varphi\in\B}\left\Vert \sum_{k\in\N}V_{k}\varphi-\sum_{k\in\N}\mathcal{R}_{k}V_{k}\varphi\right\Vert _{L_{\nu}^{q}}\leq\sup_{\varphi\in\B}\sum_{k\in\N}\left\Vert V_{k}\varphi-\mathcal{R}_{k}V_{k}\varphi\right\Vert _{L_{\nu}^{q}}\\
 & \ll\sum_{k\in\N}t_{k}^{1/q}d_{n_{k}}^{L}\left(\b_{p}^{\kappa M_{k}},\ell_{q}^{\kappa M_{k}}\right)
\end{align*}
and in the same way for $q=\infty$
\begin{align*}
d_{n}^{L}\left(\mathscr{B}W^{\s,p},L_{\nu}^{\infty}\right) & \ll\sum_{k\in\N}2^{-k\rr}d_{n_{k}}^{L}\left(\b_{p}^{\kappa M_{k}},\ell_{\infty}^{\kappa M_{k}}\right).
\end{align*}

Now, for $\star=K$, by the definition of $d_{k}^{K}\coloneqq d_{n_{k}}^{K}\left(\b_{p}^{\kappa M_{k}},\ell_{q}^{\kappa M_{k}}\right)$,
we note that for each $k\in\N$ and $n_{k}\in\N_{0}$ there exists
an at most $n_{k}$-dimensional subspace $X_{k}<_{n_{k}}\ell_{q}^{\kappa M_{k}}$
such that for each $a\in\ell_{p}^{\kappa M_{k}}$ we find $b\in X_{k}$
with 
\begin{equation}
\left\Vert a-b\right\Vert _{\ell_{q}^{\kappa M_{k}}}\leq d_{k}^{K}\left\Vert a\right\Vert _{\ell_{p}^{\kappa M_{k}}}.\label{eq:KolmogorovFiniteDimIneq}
\end{equation}
Now, with $\tilde{\mathcal{X}}_{k}\coloneqq\left\{ \sum_{Q\in\boldsymbol{P}_{k}}b_{r,Q}\tilde{\varphi}_{r,Q}:\left(b_{r,Q}\right)\in X_{k}\right\} <_{n_{k}}L_{\nu}^{q}$
we have that the subspace $\mathcal{X}_{n}\coloneqq\spann\left(\mathcal{\tilde{X}}_{k}:k\in\N\right)$
is at most of dimension $n=\sum_{k\in\N}n_{k}$ and 
\[
d_{n}^{K}\left(\mathscr{B}W^{\s,p},L_{\nu}^{q}\right)=\inf_{W<_{n}L_{\nu}^{q}}\sup_{\varphi\in\mathscr{B}W^{\s,p}}\inf_{\psi\in W}\left\Vert \varphi-\psi\right\Vert _{L_{\nu}^{q}}\leq\sup_{\varphi\in\mathscr{B}W^{\s,p}}\inf_{\psi\in\mathcal{X}_{n}}\left\Vert \varphi-\psi\right\Vert _{L_{\nu}^{q}}.
\]
To find an upper bound, we fix an arbitrary $\varphi\in\mathscr{B}W^{\s,p}$
and $c=\left(c_{r,Q}\right)\in\ell_{q}^{\kappa M_{k}}$ such that
$V_{k}\varphi=\sum_{Q\in\boldsymbol{P}_{k}}\sum_{r=1}^{\kappa}c_{r,Q}\tilde{p}_{r,Q}\in\mathcal{P}_{k}.$
Observe that for $c\in\ell_{q}^{\kappa M_{k}}$ we find $\left(d_{r,Q}\right)\in X_{k}$
such that \prettyref{eq:KolmogorovFiniteDimIneq} holds. Thus, with
$\psi_{k}=\sum_{Q\in\boldsymbol{P}_{k}}\sum_{r}d_{r,Q}\varphi_{r,Q}\in\tilde{\mathcal{X}}_{k}$
and \prettyref{lem:BasicEstimateV_kf} we conclude  for $q<\infty$
\begin{align*}
\left\Vert V_{k}\varphi-\psi_{k}\right\Vert _{L_{\nu}^{q}} & =\left\Vert \sum_{Q\in\boldsymbol{P}_{k}}\sum_{r}\left(c_{r,Q}-d_{r,Q}\right)\tilde{p}_{r,Q}\right\Vert _{L_{\nu}^{q}}\\
 & \leq\max_{Q\in\boldsymbol{P}_{k}}\J\left(Q\right)^{1/q}\left(\sum_{Q\in\boldsymbol{P}_{k}}\left\Vert \sum_{r=1}^{\kappa}\left(c_{r,Q}-d_{r,Q}\right)p_{r,Q}\right\Vert _{\mathcal{C}(Q)}^{q}\right)^{1/q}\\
 & \ll t_{k}^{1/q}\left(\sum_{Q\in\boldsymbol{P}_{k}}\sum_{r}\left\vert c_{r,Q}-d_{r,Q}\right\vert ^{q}\right)^{1/q}\leq t_{k}^{1/q}d_{k}^{K}\left\Vert c_{r,Q}\right\Vert _{\ell_{p}^{\kappa M_{k}}}\ll t_{k}^{1/q}d_{k}^{K}\left\Vert \varphi\right\Vert _{L^{\s,p}}.
\end{align*}
For $q=\infty$ we have
\begin{align*}
\left\Vert V_{k}\varphi-\psi_{k}\right\Vert _{L_{\nu}^{q}} & =\left\Vert \sum_{Q\in\boldsymbol{P}_{k}}\sum_{r}\left(c_{r,Q}-d_{r,Q}\right)\tilde{p}_{r,Q}\right\Vert _{L_{\nu}^{\infty}}\leq\max_{Q\in\boldsymbol{P}_{k}}\Lambda\left(Q\right)^{\rr/m}\left\Vert \sum_{r=1}^{\kappa}\left(c_{r,Q}-d_{r,Q}\right)p_{r,Q}\right\Vert _{\mathcal{C}(Q)}\\
 & \ll2^{-k\rr}\left(\max_{Q\in\boldsymbol{P}_{k}}\max_{r}\left\vert c_{r,Q}-d_{r,Q}\right\vert \right)=2^{-k\rr}\left\Vert c_{r,Q}-d_{r,Q}\right\Vert _{\ell_{\infty}^{\kappa M_{k}}}\\
 & \leq2^{-k\rr}d_{k}^{K}\left\Vert c_{r,Q}\right\Vert _{\ell_{p}^{\kappa M_{k}}}\ll2^{-k\rr}d_{k}^{K}\left\Vert \varphi\right\Vert _{L^{\s,p}}.
\end{align*}

Since for each $\varphi=\sum_{k}V_{k}\varphi$ we find such a sequence
$\left(\psi_{k}\right)_{k\in\N}$ with $\psi_{k}\in\tilde{\mathcal{X}}_{k}$
and $\sum_{k}\psi_{k}\in\mathcal{X}_{n}$, it follows for the Kolmogorov
widths for $1\leq q\leq\infty$

\begin{align*}
d_{n}^{K}\left(\mathscr{B}W^{\s,p},L_{\nu}^{q}\right) & \leq\sup_{\varphi\in\mathscr{B}W^{\s,p}}\left\Vert \sum_{k\in\N}V_{k}\varphi-\sum_{k\in\N}\psi_{k}\right\Vert _{L_{\nu}^{q}}\leq\sup_{\varphi\in\mathscr{B}W^{\s,p}}\sum_{k\in\N}\left\Vert V_{k}\varphi-\psi_{k}\right\Vert _{L_{\nu}^{q}}\\
 & \ll\sup_{\varphi\in\mathscr{B}W^{\s,p}}\sum_{k\in\N}T_{k}d_{k}^{K}\left\Vert \varphi\right\Vert _{L^{\s,p}}=\sum_{k\in\N}T_{k}d_{n_{k}}^{K}\left(\b_{p}^{\kappa M_{k}},\ell_{q}^{\kappa M_{k}}\right).
\end{align*}

For $\star=G$, first note that by \prettyref{lem:-Gelfand_extension}
we may consider $\left(W^{\s,p},\left\Vert \,\cdot\,\right\Vert _{L_{\nu}^{q}}\right)$
instead of $L_{\nu}^{q}$ in the definition of the Gel{\cprime}fand
widths. By the definition of $d_{k}^{G}\coloneqq d_{n_{k}}^{G}\left(\b_{p}^{\kappa M_{k}},\ell_{q}^{\kappa M_{k}}\right)$,
$k\in\N$, we find linear forms $e_{1}^{k},\ldots,e_{n_{k}}^{k}\in\left(\ell_{q}^{\kappa M_{k}}\right)'$
such that for all 
\[
b\in\left\{ a\in\ell_{q}^{\kappa M_{k}}:e_{i}^{k}\left(a\right)=0,i=1,\ldots,n_{k}\right\} <^{n_{k}}\ell_{q}^{\kappa M_{k}}
\]
 we have
\[
\left\Vert b\right\Vert _{\ell_{q}^{\kappa M_{k}}}\leq d_{k}^{G}\left\Vert b\right\Vert _{\ell_{p}^{\kappa M_{k}}}.
\]
Now, for $i=1,\ldots,n_{k}$ and $k\in\N$, we define linear functionals
on $W^{\s,p}$ by
\[
\tilde{e}_{i}^{k}\coloneqq e_{i}^{k}\circ I_{k}^{-1}\circ V_{k}:W^{\s,p}\to\R.
\]
This gives rise to the subspace 
\[
\mathcal{X}_{n}\coloneqq\left\{ \varphi\in W^{\s,p}:\tilde{e}_{i}^{k}\varphi=0,i=1,\ldots,n_{k},k\in\N\right\} <^{n}L_{\nu}^{q}\cap W^{\s,p}
\]
of co-dimension at most $n\coloneqq\sum n_{k}$. Thus, for $\varphi\in\mathcal{X}_{n}$
with $V_{k}\varphi=\sum_{Q\in\boldsymbol{P}_{k}}\sum_{r=1}^{\kappa}c_{r,Q}\tilde{p}_{r,Q}$,
$k\in\N$, we have 
\[
0=\tilde{e}_{i}^{k}\varphi=e_{i}^{k}\circ I_{k}^{-1}\circ V_{k}\varphi=e_{i}^{k}\circ I_{k}^{-1}\left(\sum_{Q\in\boldsymbol{P}_{k}}\sum_{r=1}^{\kappa}c_{r,Q}\tilde{p}_{r,Q}\right)=e_{i}^{k}\left(c_{r,Q}\right)_{r,Q},
\]
which gives, as in the previous estimations for $\star=L,K$, 
\[
\left\Vert V_{k}\varphi\right\Vert _{L_{\nu}^{q}}\ll T_{k}\left\Vert \left(c_{r,Q}\right)\right\Vert _{\ell_{q}^{\kappa M_{k}}}\leq T_{k}d_{k}^{G}\left\Vert \left(c_{r,Q}\right)\right\Vert _{\ell_{p}^{\kappa M_{k}}}\ll T_{k}d_{k}^{G}\left\Vert \varphi\right\Vert _{L^{\s,p}}.
\]
The claim follows then for the Gel{\cprime}fand widths by observing
\begin{align*}
d_{n}^{G}\left(\mathscr{B}W^{\s,p},L_{\nu}^{q}\right) & =\inf_{W<^{n}W^{\s,p}}\sup_{\varphi\in\mathscr{B}W^{\s,p}\cap W}\left\Vert \varphi\right\Vert _{L_{\nu}^{q}}=\inf_{W<^{n}W^{\s,p}}\sup_{\varphi\in\mathscr{B}W^{\s,p}\cap W}\left\Vert \sum_{k\in\N}V_{k}\varphi\right\Vert _{L_{\nu}^{q}}\\
 & \leq\sup_{\varphi\in\mathscr{B}W^{\s,p}\cap\mathcal{X}_{n}}\sum_{k\in\N}\left\Vert V_{k}\varphi\right\Vert _{L_{\nu}^{q}}\ll\sum_{k\in\N}T_{k}d_{n_{k}}^{G}\left(\b_{p}^{\kappa M_{k}},\ell_{q}^{\kappa M_{k}}\right).
\end{align*}
\end{proof}
\begin{rem}
\label{rem:d_0 and d_M}Note that for $\star\in\left\{ K,G,L\right\} $
we have 
\begin{equation}
d_{0}^{\star}\left(\b_{p}^{M},\ell_{q}^{M}\right)=\sup_{x\in\b_{p}^{M}}\left\Vert x\right\Vert _{\ell_{q}^{M}}=M^{\left(1/q-1/p\right)_{+}}\label{eq:Norm_p_q}
\end{equation}
using \prettyref{eq:H=0000F6lderFiniteDim} and by the definition
of the widths we find for all $M\in\N$
\[
d_{M}^{\star}\left(\b_{p}^{M},\ell_{q}^{M}\right)=0.
\]
\end{rem}

Now we are in the position to give the necessary improved upper bound. 
\begin{prop}
\label{prop:ImprovedupperBoundCase} We have for $\star\in\left\{ K,G,L\right\} $
the following upper bounds: 
\begin{enumerate}
\item If $p\leq q$ and $\gamma\geq0$ such that 
\[
d_{n}^{\star}\left(\b_{p}^{2n},\ell_{q}^{2n}\right)\ll n^{-\gamma},
\]
then 
\begin{align*}
\uAO_{\star}\left(\mathscr{B}W^{\s,p},L_{\nu}^{q}\right) & \leq-\overline{S}_{\r}-\gamma
\end{align*}
and if $q=\infty$, then
\begin{align*}
\lAO_{\star}\left(\mathscr{B}W^{\s,p},L_{\nu}^{\infty}\right) & \leq-\frac{\rr}{\underline{\dim}_{M}\left(\nu\right)}-\gamma.
\end{align*}
\item If $q\leq p$, then 
\[
\uAO_{\star}\left(\mathscr{B}W^{\s,p},L_{\nu}^{q}\right)\leq-\overline{S}_{\r}+\frac{1}{q}-\frac{1}{p}.
\]
\end{enumerate}
\end{prop}

\begin{proof}
\emph{ ad (1): }For $\star\in\left\{ K,G,L\right\} $ set $\lambda\coloneqq\left\lceil 1+\gamma/\overline{S}_{\r}\right\rceil \geq1$
and for $N\in\N$ define, seperately for the cases $q=\infty$ and
$q<\infty$, 
\[
j_{k}\coloneqq\begin{cases}
k, & \text{for }1\leq k\leq N,\\
\lambda k, & \text{for }k>N,
\end{cases}\;\left(t_{k}\right)\coloneqq\left(2^{-j_{k}}\right)_{k\in\N}
\]
with $M_{k}$ and $T_{k}$ accordingly. Further, let $\mathbf{n}(N)\coloneqq\sum_{k}n_{k}\in\N$
with 
\[
n_{k}\coloneqq n_{k}\left(N\right)\coloneqq\begin{cases}
\left\lceil \kappa M_{k}/2\right\rceil , & \text{for }k=N,\\
\kappa M_{k}, & \text{for }k=1,\ldots,N-1,\\
0, & \text{for }k>N.
\end{cases}
\]
With this choice, \prettyref{eq:boundOnCardinalityOFCubes} is satisfied
and we find $\mathbf{n}(N)\coloneqq\sum_{k=1}^{N-1}\kappa M_{k}+\left\lceil \kappa M_{N}/2\right\rceil \leq\kappa NM_{N}$
and with \prettyref{rem:d_0 and d_M} 
\[
d_{k}^{\star}\coloneqq d_{n_{k}}^{\star}\left(\b_{p}^{\kappa M_{k}},\ell_{q}^{\kappa M_{k}}\right)\ll\begin{cases}
\left(\kappa M_{k}/2\right)^{-\gamma}, & \text{for }k=N,\\
0, & \text{for }k=1,\ldots,N-1,\\
1, & \text{for }k>N.
\end{cases}
\]
Then \prettyref{prop:UpperBound_Discretisation} gives 
\begin{align*}
d_{\mathbf{n}(N)}^{\star}\left(\mathscr{B}W^{\s,p},L_{\nu}^{q}\right) & \ll\sum_{k\in\N}T_{k}d_{k}^{\star}\ll T_{N}\left(\kappa M_{N}/2\right)^{-\gamma}+\sum_{k>N}T_{\lambda k}\ll T_{N}M_{N}^{-\gamma}+T_{\lambda N}.
\end{align*}
Hence, for all $n\in\N$ with $\mathbf{n}(N)\leq n<\mathbf{n}(N+1)$
we have
\[
\frac{\log\left(d_{n}^{\star}\left(\mathscr{B}W^{\s,p},L_{\nu}^{q}\right)\right)}{\log\left(n\right)}\leq\frac{\log\left(d_{\mathbf{n}(N)}^{\star}\left(\mathscr{B}W^{\s,p},L_{\nu}^{q}\right)\right)}{\log\left(\mathbf{n}(N+1)\right)}.
\]
This gives 
\begin{align*}
 & \negthickspace\negthickspace\negthickspace\negthickspace\uAO_{\star}\left(\mathscr{B}W^{\s,p},L_{\nu}^{q}\right)\\
 & =\limsup_{n\to\infty}\frac{\log\left(d_{n}^{\star}\left(\mathscr{B}W^{\s,p},L_{\nu}^{q}\right)\right)}{\log\left(n\right)}\leq\limsup_{N\to\infty}\frac{\log\left(d_{\mathbf{n}(N)}^{\star}\left(\mathscr{B}W^{\s,p},L_{\nu}^{q}\right)\right)}{\log\left(\mathbf{n}(N+1)\right)}\\
 & \leq\limsup_{N\to\infty}\max\left\{ \frac{-\gamma\log\left(M_{N}\right)+\log\left(T_{N}\right)}{\log\left(\mathbf{n}(N+1)\right)},\frac{\log\left(T_{\lambda N}\right)}{\log\left(\mathbf{n}(N+1)\right)}\right\} \\
 & =\limsup_{N\to\infty}\frac{-\gamma\log\left(M_{N}\right)+\log\left(T_{N}\right)}{\log\left(\kappa\left(N+1\right)M_{N+1}\right)}\\
 & =-\overline{S}_{\r}-\gamma.
\end{align*}
If in the above estimate we consider the lower limit instead of the
upper limit for the special case $q=\infty$ and use the definition
of $T_{N}=2^{-N\rr}$ and $M_{N}=\card\mathcal{D}_{N}$, then we obtain
\[
\lAO_{\star}\left(\mathscr{B}W^{\s,p},L_{\nu}^{\infty}\right)=-\rr/\underline{\dim}_{M}\left(\nu\right)-\gamma.
\]

\emph{ad (2):} First note that by \prettyref{lem:GeneralBounds dn,dn}
it is sufficient to consider the case $\star=L$. For $N\in\N$, $t_{k}\coloneqq2^{-k}$,
respectively $j_{k}\coloneqq k$, $k\in\N$, and 
\[
n_{k}=n_{k}\left(N\right)\coloneqq\begin{cases}
\kappa M_{k}, & \text{for }k=1,\ldots,N,\\
0, & \text{for }k>N
\end{cases}
\]
we have, on the one hand, that \prettyref{eq:boundOnCardinalityOFCubes}
holds and $\mathbf{n}\left(N\right)\coloneqq\sum n_{k}\left(N\right)\ll\kappa NM_{N}.$
On the other hand, using \prettyref{prop:UpperBound_Discretisation}
combined with $d_{0}^{L}\left(\b_{p}^{n},\ell_{q}^{n}\right)=n^{1/q-1/p}$,
we obtain for $\epsilon>0$, $N\in\N$ sufficiently large, and $q<\infty$

\begin{align*}
d_{n}^{L}\left(\mathscr{B}W^{\s,p},L_{\nu}^{q}\right) & \ll\sum_{k=N+1}^{\infty}2^{-k/q}\left(\kappa M_{k}\right)^{\left(1/q-1/p\right)}=\sum_{k=N+1}^{\infty}2^{-k/q+k\left(1/q-1/p\right)\left(\log\left(\kappa M_{k}\right)/\left(k\log2\right)\right)}\\
 & \leq\sum_{k=N+1}^{\infty}2^{k\left(-1/q+\left(1/q-1/p\right)\left(s_{\r}+\epsilon\right)\right)}\\
 & \leq2^{N\left(-1/q+\left(1/q-1/p\right)\left(s_{\r}+\epsilon\right)\right)}\sum_{k=N+1}^{\infty}2^{-\left(k-N\right)\left(1/q-\left(1/q-1/p\right)\left(s_{\r}+\epsilon\right)\right)}\\
 & \ll2^{N\left(-1/q+\left(1/q-1/p\right)\left(s_{\r}+\epsilon\right)\right)}.
\end{align*}
 This gives
\begin{align*}
\uAO_{L}\left(\mathscr{B}W^{\s,p},L_{\nu}^{q}\right) & \leq\limsup_{N}\frac{-\left(N/q\right)\log2+\left(1/q-1/p\right)\left(s_{\r}+\epsilon\right)\log2}{\log N+\log\kappa+\log M_{N}}\\
 & =-\frac{1}{q\cdot s_{\r}}+\left(\frac{1}{q}-\frac{1}{p}\right)\left(\frac{s_{\r}+\epsilon}{s_{\r}}\right)
\end{align*}
and for $\epsilon\searrow0$ the claim follows. For $q=\infty$ we
have 
\begin{align*}
d_{n}^{L}\left(\mathscr{B}W^{\s,p},L_{\nu}^{\infty}\right) & \ll\sum_{k=N+1}^{\infty}2^{-k\rr}\left(\kappa M_{k}\right)^{\left(1/q-1/p\right)}=\sum_{k=N+1}^{\infty}2^{-k\rr+k\left(1/q-1/p\right)\left(\log\left(\kappa M_{k}\right)/\left(k\log2\right)\right)}\\
 & \leq\sum_{k=N+1}^{\infty}2^{k\left(-\rr+\left(1/q-1/p\right)\left(\overline{\dim}_{M}\left(\nu\right)+\epsilon\right)\right)}\\
 & \leq2^{N\left(-\rr+\left(1/q-1/p\right)\left(s_{\r}+\epsilon\right)\right)}\sum_{k=N+1}^{\infty}2^{-\left(k-N\right)\left(\rr-\left(1/q-1/p\right)\left(\overline{\dim}_{M}\left(\nu\right)+\epsilon\right)\right)}\\
 & \ll2^{N\left(-\rr+\left(1/q-1/p\right)\left(\overline{\dim}_{M}\left(\nu\right)+\epsilon\right)\right)}.
\end{align*}
 This gives
\begin{align*}
\uAO_{L}\left(\mathscr{B}W^{\s,p},L_{\nu}^{\infty}\right) & \leq\limsup_{N}\frac{-\left(N\rr\right)\log2+\left(1/q-1/p\right)\left(\overline{\dim}_{M}\left(\nu\right)+\epsilon\right)\log2}{\log N+\log\kappa+\log M_{N}}\\
 & =-\frac{\rr}{\overline{\dim}_{M}\left(\nu\right)}+\left(\frac{1}{q}-\frac{1}{p}\right)\left(\frac{\overline{\dim}_{M}\left(\nu\right)+\epsilon}{\overline{\dim}_{M}\left(\nu\right)}\right)
\end{align*}
and for $\epsilon\searrow0$, the assertion follows.
\end{proof}
\begin{proof}[Proof of upper bounds in \prettyref{thm:MAIN}]
The upper bounds provided in \prettyref{prop:ImprovedupperBoundCase}
combined with \prettyref{lem:KolmogorovFiniteDim}, \prettyref{cor:GelfandFiniteDim}
and \prettyref{cor:LinearFiniteDim} provide us with all upper bounds
as stated in our main theorem \prettyref{thm:MAIN} for the Kolmogorov
as well as in \prettyref{rem:Main} (4) for the Gel{\cprime}fand
and linear upper approximation orders. For $q=\infty$ the upper bound
for the lower approximation orders as stated in \eqref{thm:MAIN_lower}
is also contained in \prettyref{prop:ImprovedupperBoundCase} (1).
\end{proof}

\section{Coarse Muiltifractal formalism and lower bounds \label{sec:Coarse-Muiltifractal-formalism}}

In this section we only consider the Kolmogorov and Gel{\cprime}fand
case since the linear approximation order is always bounded from below
by the maximum of the other two approximation orders.

Recall from the introduction the definition of the lower and upper
optimise coarse multifractal dimension as stated in \prettyref{eq:upperlower optimal MFdim}
and that by \cite{KN2022,KN2022b} we know that $\overline{\mathcal{F}}_{\r}=s_{\r}$.
It therefore suffices to determine a lower bound for the upper approximation
orders in terms of $\overline{\mathcal{F}}_{\r}$.

As mentioned in the outline, we will give the lower bounds on the
approximation orders with respect to $\BB$ and it will be convenient
to equip the space with the equivalent norm $\left\Vert \,\cdot\,\right\Vert _{L^{\s,p}}$.
The following basic norm equality will be crucial throughout this
section.
\begin{lem}
\label{lem:support-scaling}For each cube $Q\in\mathcal{D}$, $\phi_{Q}:\Q\to Q$
as defined in \prettyref{eq:Def_phi_Q} and for all $u\in W_{p}^{\alpha}\left(\Q\right)$
we have $u_{Q}\coloneqq u\circ\phi_{Q}^{-1}\in W_{p}^{\alpha}\left(Q\right)$
and 
\[
\left\Vert u_{Q}\right\Vert _{L^{\s,p}\left(Q\right)}=\Lambda\left(Q\right)^{-\rr/m}\left\Vert u\right\Vert _{L^{\s,p}}.
\]
 
\end{lem}

\begin{proof}
For $Q\in\mathcal{D}_{n}$, recall the definition of $\phi_{Q}:\R^{m}\to\R^{m},x\mapsto c_{Q}x+b$
with $c_{Q}=\left(\Lambda\left(Q\right)\right)^{1/m}=2^{-n}$. This
gives for $u_{Q}\coloneqq u\circ\phi_{Q}^{-1}$, using chain rule
and substitution, 
\begin{align*}
\left\Vert u_{Q}\right\Vert _{L^{\s,p}\left(Q\right)} & =\left(\int_{\R^{m}}\left(\sum_{\lvert k\rvert=\s}\left|D^{k}\left(u\circ\phi_{Q}^{-1}\right)\right|^{2}\right)^{p/2}\d\Lambda\right)^{1/p}\\
 & =\left(\int_{\R^{m}}\left\vert \det\left(\left(\phi_{Q}^{-1}\right)'\right)\right\vert ^{-1}\left(c_{Q}^{-2\s}\sum_{\lvert k\rvert=\s}\lvert D^{k}u\rvert^{2}\right)^{p/2}\d\Lambda\right)^{1/p}\\
 & =\left(\int_{\R^{m}}c_{Q}^{-\s p+m}\left(\sum_{\lvert k\rvert=\s}\lvert D^{k}u\rvert^{2}\right)^{p/2}\d\Lambda\right)^{1/p}=\Lambda\left(Q\right)^{-\rr/m}\left\Vert u\right\Vert _{L^{\s,p}}.
\end{align*}
For $p=\infty$ we get by the chain rule
\begin{align*}
\left\Vert u_{Q}\right\Vert _{L^{\s,\infty}} & =\left\Vert \left(\sum_{\lvert k\rvert=\s}\left|D^{k}\left(u\circ\phi_{Q}^{-1}\right)\right|^{2}\right)^{1/2}\right\Vert _{L^{\infty}}=\left\Vert \left(c_{Q}^{-2\s}\sum_{\lvert k\rvert=\s}\left|\left(D^{k}u\right)\circ\phi_{Q}^{-1}\right|^{2}\right)^{1/2}\right\Vert _{L^{\infty}}\\
 & =c_{Q}^{-\s}\left\Vert \left(\sum_{\lvert k\rvert=\s}\left|\left(D^{k}u\right)\right|^{2}\right)^{1/2}\right\Vert _{L^{\infty}}=c_{Q}^{-\s}\left\Vert u\right\Vert _{L^{\s,\infty}}=\Lambda\left(Q\right)^{-\rr/m}\left\Vert u\right\Vert _{L^{\s,\infty}}.
\end{align*}
\end{proof}
\begin{notation}
For a cube $Q\in\mathcal{D}$ and $r>0$ we write $\left\langle Q\right\rangle _{r}$
for the cube of side length $r\Lambda\left(Q\right)^{1/m}$ that is
parallel to the axis and has the same midpoint as $Q$.
\end{notation}

The following construction is standard. 
\begin{lem}
\label{lem:mollifier}There exists $u\in C_{c}^{\infty}$that fulfils
the following properties:
\end{lem}

\begin{enumerate}
\item $0\leq u\leq1$,
\item $\supp(u)\subset\Q$ and
\item $u=1$ on $\langle\Q\rangle_{1/3}$.
\end{enumerate}
\begin{proof}
We make use of the \emph{Friedrich mollifier} on $\R^{m}$ given by
\[
x\mapsto\psi(x)\coloneqq\begin{cases}
\exp\left(1/\left(\left\Vert x\right\Vert _{\ell_{2}^{m}}^{2}-1\right)\right), & \left\Vert x\right\Vert _{\ell_{2}^{m}}<1,\\
0, & \text{else }
\end{cases}
\]
and for $\epsilon>0$ setting 
\[
\psi_{\epsilon}:x\mapsto\left(\epsilon^{m}\int\psi\d\Lambda\right)^{-1}\psi\left(\frac{x}{\epsilon}\right).
\]
Then all properties are fulfilled by the convolution $u\coloneqq\psi_{1/4}\star\mathds{1}_{\langle\Q\rangle_{1/2}}$. 
\end{proof}
\begin{lem}[Well-separated subfamilies]
\label{lem:disjoint-3-cubes}For each $n\in\N$ and each set of cubes
$N\subset\mathcal{D}_{n}$, such that 
\[
\sup_{Q\in\mathcal{D}_{n}\setminus N}\nu\left(Q\right)\leq\inf_{Q\in N}\nu\left(Q\right),
\]
there exists a subset of $N$ denoted by $\hat{N}$ with the following
properties:
\begin{enumerate}
\item $\langle\mathring{Q}\rangle_{3}\cap\langle\mathring{Q}'\rangle_{3}=\emptyset$
for all $Q,Q'\in\hat{N}$ with $Q\neq Q'$,
\item $\card\hat{N}\geq\lfloor\card\left(N\right)/5^{m}\rfloor$ and 
\item for all $Q\in\hat{N}$ and all its neighbouring cubes of the same
side length $Q'\in\mathfrak{N}\left(Q\right)\coloneqq\left\{ C\in\mathcal{D}_{-\log_{2}(\Lambda(Q))/m}:\overline{C}\cap\overline{Q}\neq\emptyset\right\} $,
we have $\nu(Q)\geq\nu(Q')$.
\end{enumerate}
\end{lem}

\begin{proof}
 We inductively construct a finite decreasing sequence $\left(D^{(0)},\ldots,D^{(k)}\right)$
of subsets of $N$ for some $k\in\N$ as follows. We set $D^{(0)}=N$.
Now, for given $D^{(i)}$ we let 
\[
D_{(i)}\coloneqq\left\{ Q\in D^{(i)}:\exists Q'\in D^{(i)}\setminus\left\{ Q\right\} :\mathring{Q}'\cap\langle\mathring{Q}\rangle_{5}\neq\emptyset\right\} .
\]
If $D_{(i)}$ is not empty, then we choose $Q\in D_{(i)}$ such that
$\nu\left(Q\right)=\max\left\{ \nu\left(Q'\right):Q'\in D_{(i)}\right\} $
and define
\[
D^{(i+1)}=\left\{ Q'\in D^{(i)}:\mathring{Q}'\cap\langle\mathring{Q}\rangle_{5}=\emptyset\right\} \cup\left\{ Q\right\} .
\]
Otherwise, if $D_{(i)}$ is empty, then our induction terminates for
$k\coloneqq i$ and we define $\hat{N}\coloneqq D^{\left(k\right)}.$
Note that there are $5^{m}-1$ cubes $Q'\in\mathcal{D}_{n}\setminus\left\{ Q\right\} $
at most that fulfil $\langle\mathring{Q}\rangle_{5}\cap\mathring{Q}'\neq\emptyset$
for all $Q\in\mathcal{D}_{n}$. Thus,
\[
\card\hat{N}\geq\lfloor\mathcal{N}/5^{m}\rfloor.
\]
The third property is clear by construction and our assumption. 
\end{proof}
Since both $\mathcal{D}_{n}$ and $N_{\r,n}\left(\a\right)$ fulfil
the assumption imposed on $N$ in \prettyref{lem:disjoint-3-cubes},
we define 
\[
Q_{n}\left(\a\right)\coloneqq\widehat{N_{\r,n}\left(\a\right)}\:\text{and }\;c_{n,\a}\coloneqq\card Q_{n}\left(\a\right),
\]
where as before $\mathcal{N}_{\r,n}\left(\a\right)\coloneqq\card N_{\r,n}\left(\a\right)$,
and 
\[
Q_{n}\coloneqq\widehat{\mathcal{D}_{n}},\:\text{and }\;c_{n}\coloneqq\card Q_{n},\:\mathcal{N}_{n}\coloneqq\card\mathcal{D}_{n}.
\]

For $n\in\N$ and we choose $u\in C_{c}^{\infty}$ with $\supp u\subset\Q$
as in \prettyref{lem:mollifier} and $u_{Q}\coloneqq u\circ\varphi_{\left\langle Q\right\rangle _{3}}^{-1}$
given by \prettyref{lem:support-scaling} for all $Q\in\mathcal{D}_{n}$.
This gives rise to the $c_{\a,n}$-dimensional, respectively $c_{n}$-dimensional,
subspace of $W_{0}^{\s,p}$
\[
\mathcal{W}_{\a,n}\coloneqq\spann\left\{ u_{Q}:Q\in Q_{n}\left(\a\right)\right\} ,\;\text{and }\:\mathcal{W}_{n}\coloneqq\spann\left\{ u_{Q}:Q\in Q_{n}\right\} ,\:\text{respect. }
\]

\begin{rem}
\label{rem:DisjointQ_n}Observe that for $\mathfrak{Q}_{n}\in\left\{ Q_{n}\left(\a\right),Q_{n}\right\} $
and $Q,Q'\in\mathfrak{Q_{n}}$ with $Q\neq Q'$ we have $\supp u_{Q}\cap\supp u_{Q'}\subset\left\langle Q\right\rangle _{3}\cap\left\langle Q'\right\rangle _{3}=\emptyset$
and since $\supp\left(\nu\right)\subset\mathring{\Q}$ we have $\left\langle Q\right\rangle _{3}\subset\Q$
for all $Q\in\mathfrak{Q_{n}}$ and $n\in\N$ large enough. By \prettyref{lem:support-scaling}
we have 
\[
\left\Vert u_{Q}\right\Vert _{L^{\s,p}\left(Q\right)}=\Lambda\left(\left\langle Q\right\rangle _{3}\right)^{-\rr/m}\left\Vert u\right\Vert _{L^{\s,p}}=\left(3\cdot2^{-n}\right)^{-\rr}\left\Vert u\right\Vert _{L^{\s,p}}.
\]
\end{rem}

\begin{lem}
\label{lem:isomorphic_spheres}Let $\a>0$ and for $u\in C_{c}^{\infty}\left(\Q\right)$
with $\left\Vert u\right\Vert _{L^{\s,p}}>0$ fix $C\coloneqq3^{\rr}/\left\Vert u\right\Vert _{L^{\s,p}}$.
Then we have
\begin{enumerate}
\item $\left(a_{Q}:Q\in Q_{n}\left(\a\right)\right)\in C2^{-n\rr}\b_{p}^{c_{\a,n}}$
if and only if $f=\sum_{Q\in Q_{n}(\a)}a_{Q}u_{Q}\in\mathscr{B}\mathcal{W}_{\a,n}$, 
\item $\left(a_{Q}:Q\in Q_{n}\right)\in C2^{-n\rr}\b_{p}^{c_{n}}$ if and
only if $f=\sum_{Q\in Q_{n}(\a)}a_{Q}u_{Q}\in\mathscr{B}\mathcal{W}_{n}$.
\end{enumerate}
\end{lem}

\begin{proof}
\emph{ad (1):} By the construction of $Q_{n}\left(\a\right)\subset N_{\r,n}\left(\a\right)$
for fixed $n\in\N$, see \prettyref{lem:disjoint-3-cubes}, we find
for $f=\sum_{Q}a_{Q}u_{Q}\in\mathcal{W}_{\a,n}$ and $p<\infty$
\begin{align*}
\left\Vert \sum_{Q\in Q_{n}\left(\a\right)}a_{Q}u_{Q}\right\Vert _{L^{\s,p}} & =\left(\int_{\Q}\left\vert \sum_{\left\vert k\right\vert =\s}\left(D^{k}\sum_{Q\in Q_{n}\left(\a\right)}a_{Q}u_{Q}\right)^{2}\right\vert ^{p/2}\d\Lambda\right)^{1/p}\\
 & =\left(\sum_{Q\in Q_{n}\left(\a\right)}\sum_{\left\vert k\right\vert =\s}\int\left\vert D^{k}a_{Q}u_{Q}\right\vert ^{p}\d\Lambda\right)^{1/p}\\
 & =\left(\sum_{Q\in Q_{n}\left(\a\right)}\left\Vert a_{Q}u_{Q}\right\Vert _{L^{\s,p}\left(Q\right)}^{p}\right)^{1/p}
\end{align*}
using the fact that the $u_{Q}$ have disjoint supports. For $p=\infty$
we get 
\begin{align*}
\left\Vert \sum_{Q\in Q_{n}\left(\a\right)}a_{Q}u_{Q}\right\Vert _{L^{\s,\infty}} & =\left\Vert \left\vert \sum_{\left\vert k\right\vert =\s}\left(D^{k}\sum_{Q\in Q_{n}\left(\a\right)}a_{Q}u_{Q}\right)^{2}\right\vert ^{1/2}\right\Vert _{L^{\infty}}\\
 & =\max_{Q\in Q_{n}\left(\a\right)}\left\Vert \left\vert \sum_{\left\vert k\right\vert =\s}\left(D^{k}a_{Q}u_{Q}\right)^{2}\right\vert ^{1/2}\right\Vert _{L^{\infty}}\\
 & =\max_{Q\in Q_{n}\left(\a\right)}\left\Vert a_{Q}u_{Q}\right\Vert _{L^{\s,\infty}\left(Q\right)}.
\end{align*}
Together with \prettyref{lem:support-scaling} this gives for $p<\infty$
\begin{align*}
\left\Vert \sum_{Q\in Q_{n}\left(\a\right)}a_{Q}u_{Q}\right\Vert _{L^{\s,p}} & =\left(\sum_{Q\in Q_{n}\left(\a\right)}\left\vert a_{Q}\right\vert ^{p}\left(\left(3\cdot2^{-n}\right)^{-\rr}\left\Vert u\right\Vert _{L^{\s,p}}\right)^{p}\right)^{1/p}\\
 & =\left(3\cdot2^{-n}\right)^{-\rr}\left\Vert u\right\Vert _{L^{\s,p}}\left\vert \left(a_{Q}\right)\right\vert _{\ell_{p}^{c_{\a,n}}}\\
 & =\left(3\cdot2^{-n}\right)^{-\rr}\left\Vert u\right\Vert _{L^{\s,p}}\left\vert \left(a_{Q}\right)\right\vert _{\ell_{p}^{c_{\a,n}}}
\end{align*}
and for $p=\infty$
\begin{align*}
\left\Vert \sum_{Q\in Q_{n}\left(\a\right)}a_{Q}u_{Q}\right\Vert _{L^{\s,\infty}} & =\max_{Q\in Q_{n}\left(\a\right)}\left\Vert a_{Q}u_{Q}\right\Vert _{L^{\s,\infty}}=\left(3\cdot2^{-n}\right)^{-\rr}\left\Vert u\right\Vert _{L^{\s,\infty}}\left\vert \left(a_{Q}\right)\right\vert _{\ell_{\infty}^{c_{\a,n}}},
\end{align*}
which shows the first claim. The second claim follows along the same
lines. 
\end{proof}
For $\mathfrak{W}_{n}\in\left\{ \mathcal{W}_{\a,n},\mathcal{W}_{n}\right\} $,
and $\mathfrak{Q}_{n}\in\left\{ Q_{n}\left(\a\right),Q_{n}\right\} $,
$n\in\N$, let us define the linear operator
\begin{equation}
\mathcal{Q}_{n}:L_{\nu}^{q}\to\left(\mathfrak{W}_{n},\left\Vert \,\cdot\,\right\Vert _{L_{\nu}^{q}}\right),\;\mathcal{Q}_{n}\left(f\right)\coloneqq\sum_{Q\in\mathfrak{Q}_{n}}\frac{\int fu_{Q}\d\nu}{\int u_{Q}^{2}\d\nu}u_{Q}.\label{eq:operatorQn}
\end{equation}
We will now estimate its operator norm.
\begin{lem}
\label{lem:operatornormQn} For the operator norm of $\mathcal{Q}_{n}$,
as defined in \prettyref{eq:operatorQn}, we have 
\[
\left\Vert \mathcal{Q}_{n}\right\Vert \leq3^{m}.
\]
 
\end{lem}

\begin{proof}
As before, let $q'$ denote the dual of $q\in\left[1,\infty\right]$.
For $f\in L_{\nu}^{q}$, $Q\in\mathcal{D}_{n}$ with $\nu\left(Q\right)>0$
we find that $fu_{Q}\in L_{\nu}^{1}$ and for $r\geq1$ we have 
\begin{align*}
\left\Vert u_{Q}\right\Vert _{L_{\nu}^{r}} & =\left(\sum_{C\in\mathfrak{N}\left(Q\right)}\int_{C}\left\vert u_{Q}\right\vert ^{r}\d\nu\right)^{1/r}\leq\card\mathfrak{N}(Q)^{1/r}\max_{C\in\mathfrak{N}\left(Q\right)}\left\Vert \1_{C}\right\Vert _{L_{\nu}^{r}}\leq3^{m/r}\nu\left(Q\right)^{1/r}
\end{align*}
with $\mathfrak{N}(Q)$ denoting the set of neighbouring cubes of
$Q$ in $\mathcal{D}_{n}$, as before. Note that we used $\nu(Q)=\max_{C\in\mathfrak{N}(Q)}\nu(C)$
by \prettyref{lem:disjoint-3-cubes} and $u_{Q}\leq1$ by \prettyref{lem:mollifier}.
Also, $\left\Vert u_{Q}\right\Vert _{L_{\nu}^{\infty}}\leq1=3^{m/\infty}\nu\left(Q\right)^{1/\infty}$.
Thus, Hölder's inequality, the fact that the cubes in $Q_{n}\left(\a\right)$,
respectively $Q_{n}$, are disjoint and $\int_{\langle Q\rangle_{3}}u_{Q}^{2}\d\nu\geq\int_{Q}1\d\nu=\nu\left(Q\right)>0$
give for $q<\infty$
\begin{align*}
\left\Vert \mathcal{Q}_{n}f\right\Vert _{L_{\nu}^{q}}^{q} & =\sum_{Q\in Q_{n}\left(\a\right)}\left\vert \frac{\int fu_{Q}\d\nu}{\int u_{Q}^{2}\d\nu}\right\vert ^{q}\left\Vert u_{Q}\right\Vert _{L_{\nu}^{q}}^{q}\leq\sum_{Q\in Q_{n}\left(\a\right)}\frac{\left\Vert f|_{\left\langle Q\right\rangle _{3}}\right\Vert _{L_{\nu}^{q}}^{q}}{\nu\left(Q\right)^{q}}\left\Vert u_{Q}\right\Vert _{L_{\nu}^{q'}}^{q}\left\Vert u_{Q}\right\Vert _{L_{\nu}^{q}}^{q}\\
 & \leq\sum_{Q\in Q_{n}\left(\a\right)}\frac{\left\Vert f|_{\left\langle Q\right\rangle _{3}}\right\Vert _{L_{\nu}^{q}}^{q}}{\nu\left(Q\right)^{q}}\left(3^{m/q'}\nu\left(Q\right)^{1/q'}3^{m/q}\nu\left(Q\right)^{1/q}\right)^{q}\leq3^{mq}\left\Vert f\right\Vert _{L_{\nu}^{q}}^{q}
\end{align*}
and for $q=\infty$
\begin{align*}
\left\Vert \mathcal{Q}_{n}f\right\Vert _{L_{\nu}^{\infty}} & =\max_{Q\in Q_{n}}\left\vert \frac{\int fu_{Q}\d\nu}{\int u_{Q}^{2}\d\nu}\right\vert \left\Vert u_{Q}\right\Vert _{L_{\nu}^{\infty}}\leq\max_{Q\in Q_{n}}\frac{\left\Vert f|_{\left\langle Q\right\rangle _{3}}\right\Vert _{L_{\nu}^{\infty}}}{\nu\left(Q\right)}\left\Vert u_{Q}\right\Vert _{L_{\nu}^{1}}\left\Vert u_{Q}\right\Vert _{L_{\nu}^{\infty}}\\
 & \leq\max_{Q\in Q_{n}}\frac{\left\Vert f|_{\left\langle Q\right\rangle _{3}}\right\Vert _{L_{\nu}^{\infty}}}{\nu\left(Q\right)}3^{m}\nu\left(Q\right)\leq3^{m}\left\Vert f\right\Vert _{L_{\nu}^{\infty}}.
\end{align*}
\end{proof}
\begin{prop}[Discretisation – lower bound]
\label{prop:discretesation_lower bound widths}For $\star\in\left\{ K,G\right\} $,
$\a>0$, $k<c_{\a,n}$, and $n\in\N$ we have for $q<\infty$ 
\[
d_{k}^{\star}\left(\BB,L_{\nu}^{q}\right)\gg2^{-\a n/q}d_{k}^{\star}\left(\b_{p}^{c_{\a,n}},\ell_{q}^{c_{\a,n}}\right)
\]
as well as for $q=\infty$ and $k<c_{n}$ 
\[
d_{k}^{\star}\left(\BB,L_{\nu}^{\infty}\right)\gg2^{-n\rr}d_{k}^{\star}\left(\b_{p}^{c_{n}},\ell_{\infty}^{c_{n}}\right).
\]
\end{prop}

\begin{proof}
We start with the case $\star=G$. For $q<\infty$, observe that for
$n$ large enough and $\left(a_{Q}\right)\in\b_{p}^{c_{\a,n}}$ we
have $\sum_{Q\in Q_{n}\left(\a\right)}C2^{-n\rr}a_{Q}u_{Q}\in\BB$
by \prettyref{lem:isomorphic_spheres}. Also, with \prettyref{lem:disjoint-3-cubes},
\prettyref{rem:DisjointQ_n}, and \prettyref{lem:support-scaling}
for $n$ large enough, we have
\begin{align*}
\left\Vert \sum_{Q\in Q_{n}\left(\a\right)}C2^{-n\rr}a_{Q}u_{Q}\right\Vert _{L_{\nu}^{q}} & \gg\left(\sum_{Q\in Q_{n}\left(\a\right)}\left(2^{-n\rr}\left\Vert a_{Q}u_{Q}\right\Vert _{L_{\nu}^{q}}\right)^{q}\right)^{^{1/q}}=\left(\sum_{Q\in Q_{n}\left(\a\right)}\left|a_{Q}\right|^{q}\Lambda\left(Q\right)^{\r/m}\left\Vert u_{Q}\right\Vert _{L_{\nu}^{q}}^{q}\right)^{^{1/q}}\\
 & \geq\left(\sum_{Q\in Q_{n}\left(\a\right)}\left|a_{Q}\right|^{q}\Lambda\left(Q\right)^{\r/m}\nu\left(Q\right)\right)^{^{1/q}}\\
 & \geq\left(\min_{Q\in Q_{n}\left(\a\right)}\J_{\r}\left(Q\right)\right)^{1/q}\left(\sum_{Q\in Q_{n}\left(\a\right)}\left|a_{Q}\right|^{q}\right)^{^{1/q}}\geq2^{-\a n/q}\left\Vert \left(a_{Q}\right)\right\Vert _{\ell_{q}^{c_{\a,n}}}.
\end{align*}
Note that the assumption $\card\left(\text{supp}(\nu)\right)=\infty$
ensures that $c_{\a,n}\to\infty$ for $n\to\infty$. Thus, for $k\in\N$
we can choose $n\in\N$ such that $2k\leq c_{\a,n}$ and we find 
\begin{align*}
d_{k}^{G}\left(\BB,L_{\nu}^{q}\right) & =\inf\left\{ \sup_{u\in U\cap\BB}\left\Vert u\right\Vert _{L_{\nu}^{q}}:U<^{n}L_{\nu}^{q}\right\} \geq\inf\left\{ \sup_{u\in U\cap\mathcal{W}_{\a,n}\cap\BB}\left\Vert u\right\Vert _{L_{\nu}^{q}}:U<^{n}L_{\nu}^{q}\right\} \\
 & \geq\inf\left\{ \sup_{u\in U\cap\mathcal{W}_{\a,n}\cap\BB}\left\Vert u\right\Vert _{L_{\nu}^{q}}:U<^{n}\mathcal{W}_{\a,n}\right\} \\
 & \gg2^{-\a n/q}\inf\left\{ \sup_{u\in U\cap\b_{p}^{c_{\a,n}}}\left\Vert u\right\Vert _{\ell_{q}^{c_{\a,n}}}:U<^{n}\ell_{q}^{c_{\a,n}}\right\} =2^{-\a n/q}d_{k}^{G}\left(\b_{p}^{c_{\a,n}},\ell_{q}^{c_{\a,n}}\right).
\end{align*}
Now for $q=\infty$ we have $\sum_{Q\in Q_{n}}C2^{-\rr}a_{Q}u_{Q}\in\mathscr{B}W_{0}^{\s,p}$
for $a=\left(a_{Q}\right)\in\b_{p}^{c_{n}}$. This gives 
\[
\left\Vert \sum_{Q\in Q_{n}}C2^{-n\rr}a_{Q}u_{Q}\right\Vert _{L_{\nu}^{\infty}}\gg2^{-n\rr}\max_{Q\in Q_{n}}\left|a_{Q}\right|\left\Vert u_{Q}\right\Vert _{L_{\nu}^{\infty}}=2^{-n\rr}\left\Vert a\right\Vert _{\ell_{\infty}^{c_{n}}}
\]
since $u_{Q}=1$ on $Q\in Q_{n}$ and $\nu\left(Q\right)>0$, i.\,e\@.
$\left\Vert u_{Q}\right\Vert _{L_{\nu}^{\infty}}=1$. We obtain
\begin{align*}
d_{k}^{G}\left(\BB,L_{\nu}^{\infty}\right) & \geq\inf\left\{ \sup_{u\in U\cap\mathcal{W}_{\a,n}\cap\BB}\left\Vert u\right\Vert _{L_{\nu}^{\infty}}:U<^{n}\mathcal{W}_{n}\right\} \\
 & \gg2^{-n\rr}\inf\left\{ \sup_{u\in U\cap\b_{p}^{c_{n}}}\left\Vert u\right\Vert _{\ell_{\infty}^{c_{n}}}:U<^{n}\ell_{\infty}^{c_{n}}\right\} =2^{-n\rr}d_{k}^{G}\left(\b_{p}^{c_{n}},\ell_{\infty}^{c_{n}}\right),
\end{align*}
which shows the claim for the Gel{\cprime}fand widths.

For the case $\star=K$ and $q<\infty$, observe that for $v\in\mathcal{W}_{\a,n}$
we have $\mathcal{Q}_{n}v=v$ and hence, according to \prettyref{lem:operatornormQn}
we have for $f\in L_{\nu}^{q}$, 
\begin{equation}
\left\Vert v-\mathcal{Q}_{n}f\right\Vert _{L_{\nu}^{q}}=\left\Vert \mathcal{Q}_{n}\left(v-f\right)\right\Vert _{L_{\nu}^{q}}\leq3^{m}\left\Vert v-f\right\Vert _{L_{\nu}^{q}}.\label{eq:KomogorovOperatorEst.}
\end{equation}
This allows the Kolmogorov widths to be estimated as follows
\begin{align}
\inf_{W<_{k}L_{\nu}^{q}}\sup_{u\in\BB}\inf_{f\in W}\left\Vert u-f\right\Vert _{L_{\nu}^{q}} & \geq\inf_{W<_{k}L_{\nu}^{q}}\sup_{v\in\mathscr{B}\mathcal{W}_{c_{\a,n}}}\inf_{f\in W}\left\Vert v-f\right\Vert _{L_{\nu}^{q}}\nonumber \\
 & \gg\inf_{W<_{k}L_{\nu}^{q}}\sup_{v\in\mathscr{B}\mathcal{W}_{\a,n}}\inf_{f\in W}\left\Vert v-\mathcal{Q}_{n}f\right\Vert _{L_{\nu}^{q}}\label{eq:KolmogorovLowerEstDiscretization}\\
 & \geq\inf_{W<_{k}\left(\mathcal{W}_{c_{\a,n}},\left\Vert \,\cdot\,\right\Vert _{L_{\nu}^{q}}\right)}\sup_{v\in\mathscr{B}\mathcal{W}_{\a,n}}\inf_{w\in W}\left\Vert v-w\right\Vert _{L_{\nu}^{q}}.\nonumber 
\end{align}
For a subspace $\tilde{W}<_{k}\left(\mathcal{W}_{\a,n},\left\Vert \,\cdot\,\right\Vert _{L_{\nu}^{q}}\right)$
we set $W\coloneqq\left\{ \left(\left(b_{Q}\right)_{Q\in Q_{n}\left(\a\right)}:\sum_{Q\in Q_{n}\left(\a\right)}b_{Q}u_{Q}\in\tilde{W}\right)\right\} $,
and with \prettyref{lem:isomorphic_spheres} and \prettyref{lem:support-scaling}
we find

\begin{align*}
\sup_{v\in\mathscr{B}\mathcal{W}_{\a,n}}\inf_{w\in\tilde{W}}\left\Vert v-w\right\Vert _{L_{\nu}^{q}} & \gg\sup_{\left(a_{Q}\right)\in\b_{p}^{c_{\a,n}}}\inf_{w\in\tilde{W}}2^{-n\rr}\left\Vert w-\sum_{Q\in Q_{n}\left(\a\right)}a_{Q}u_{Q}\right\Vert _{L_{\nu}^{q}}\\
 & =\sup_{\left(a_{Q}\right)\in\b_{p}^{c_{\a,n}}}\inf_{\left(b_{Q}\right)\in W}2^{-n\rr}\left\Vert \sum_{Q\in Q_{n}\left(\a\right)}\left(a_{Q}-b_{Q}\right)u_{Q}\right\Vert _{L_{\nu}^{q}}\\
 & \geq\sup_{a\in\b_{p}^{c_{\a,n}}}\inf_{b\in W}\left(\sum_{Q_{n}\left(\a\right)}\Lambda\left(Q\right)^{\r}\left\Vert \left(a_{Q}-b_{Q}\right)u_{Q}\right\Vert _{L_{\nu}^{q}}^{q}\right)^{1/q}\\
 & \geq\sup_{a\in\b_{p}^{c_{\a,n}}}\inf_{b\in W}\min_{Q\in Q_{n}\left(\a\right)}\J_{\r}\left(Q\right)^{1/q}\left(\sum_{Q_{n}\left(\a\right)}\left|a_{Q}-b_{Q}\right|^{q}\right)^{1/q}\\
 & \geq\sup_{a\in\b_{p}^{c_{\a,n}}}\inf_{b\in W}2^{-n\a/q}\left\Vert \left(a_{Q}-b_{Q}\right)\right\Vert _{\ell_{q}^{c_{\a,n}}}
\end{align*}
since for all $Q\in Q_{n}\left(\a\right)$ we have $\Lambda\left(Q\right)^{-\rr/m}\nu\left(Q\right)^{1/q}=\mathfrak{J}\left(Q\right)^{1/q}\geq2^{-n\a/q}$.
Therefore,
\begin{align*}
d_{k}^{K}\left(\BB,L_{\nu}^{q}\right) & \geq d_{k}^{K}\left(\mathscr{B}\mathcal{W}_{\a,n},L_{\nu}^{q}\right)\gg d_{k}^{K}\left(\mathscr{B}\mathcal{W}_{\a,n},\left(\mathcal{W}_{\a,n},\left\Vert \cdot\right\Vert _{L_{\nu}^{q}}\right)\right)\\
 & \gg2^{-n\a/q}d_{k}\left(\b_{p}^{c_{\a,n}},\ell_{q}^{c_{\a,n}}\right).
\end{align*}
Finally, for $q=\infty$ we also have $Q_{n}v=v$ for $v\in\mathcal{W}_{n}$
and thus \prettyref{eq:KomogorovOperatorEst.} and \prettyref{eq:KolmogorovLowerEstDiscretization}
still hold in this case. Using this and \ref{lem:isomorphic_spheres},
we find
\begin{align*}
\sup_{\left(a_{Q}\right)\in\b_{p}^{c_{n}}}\inf_{w\in\tilde{W}}2^{-n\rr}\left\Vert w-\sum_{Q\in Q_{n}}a_{Q}u_{Q}\right\Vert _{L_{\nu}^{\infty}} & =\sup_{\left(a_{Q}\right)\in\b_{p}^{c_{n}}}\inf_{\left(b_{Q}\right)\in W}2^{-n\rr}\left\Vert \sum_{Q\in Q_{n}}\left(a_{Q}-b_{Q}\right)u_{Q}\right\Vert _{L_{\nu}^{\infty}}\\
 & \geq\sup_{\left(a_{Q}\right)\in\b_{p}^{c_{n}}}\inf_{\left(b_{Q}\right)\in W}2^{-n\rr}\max_{Q\in Q_{n}}\left|a_{Q}-b_{Q}\right|
\end{align*}
for $\tilde{W}<_{k}\left(\mathcal{W}_{n},\left\Vert \,\cdot\,\right\Vert _{L_{\nu}^{q}}\right)$
and $W\coloneqq\left\{ \left(\left(b_{Q}\right)_{Q\in Q_{n}\left(\a\right)}:\sum_{Q\in Q_{n}\left(\a\right)}b_{Q}u_{Q}\in\tilde{W}\right)\right\} $.
Applying the infimum over such subspaces $\tilde{W}$ shows the claim.
\end{proof}
Now we are in the position to state the lower bounds in question.
\begin{prop}
\label{prop:AllLowerBounds} Suppose that for $\gamma\in\R$, $\star\in\left\{ K,G\right\} $,
we have $d_{n}^{\star}\left(\b_{p}^{2n},\ell_{q}^{2n}\right)\gg n^{-\gamma}$,
then 
\begin{enumerate}
\item ${\displaystyle -\overline{S}_{\r}-\gamma\leq\uAO_{\star}\left(\BB,L_{\nu}^{q}\right)}$
and 
\item ${\displaystyle -\underline{S}_{\r}-\gamma\leq\lAO_{\star}\left(\BB,L_{\nu}^{q}\right)}$.
\end{enumerate}
\end{prop}

\begin{proof}
\emph{ad (1):} For $q<\infty$, consider the sequence $\left(k_{n}\right)_{n\in\N}\coloneqq\left(\left\lfloor c_{\a,n}/2\right\rfloor \right)_{n\in\N}$
and observe that $\left\lfloor c_{\a,n}/2\right\rfloor \asymp\mathcal{N}_{\r,n}\left(\a\right)$
for $\a>0$. By the definition of the upper approximation orders and
\prettyref{prop:discretesation_lower bound widths}, we have
\begin{align*}
\uAO_{\star}\left(\BB,L_{\nu}^{q}\right) & =\limsup_{n\to\infty}\frac{\log\left(d_{n}^{\star}\left(\BB,L_{\nu}^{q}\right)\right)}{\log(n)}\geq\limsup_{n\to\infty}\frac{\log\left(d_{k_{n}}^{\star}\left(\BB,L_{\nu}^{q}\right)\right)}{\log(k_{n})}\\
 & \gg\limsup_{n\to\infty}\frac{\log\left(2^{-\a n/q}d_{k_{n}}^{\star}\left(\b_{p}^{c_{\a,n}},\ell_{q}^{c_{\a,n}}\right)\right)}{\log(k_{n})}\\
 & \gg\limsup_{n\to\infty}\frac{-\a\log\left(2^{n}\right)}{q\log\left(\mathcal{N}_{\r,n}\left(\a\right)\right)}-\frac{\gamma\log\left(k_{n}\right)}{\log\left(k_{n}\right)}=\frac{-\a}{q\overline{F}_{\r}\left(\a\right)}-\gamma.
\end{align*}
Taking the supremum over all $\a>0$ yields
\[
\uAO_{\star}\left(\BB,L_{\nu}^{q}\right)\geq\sup_{\a>0}\frac{-\a}{q\overline{F}_{\r}\left(\a\right)}-\gamma=\frac{-1}{q\overline{\mathcal{F}}_{\r}}-\gamma.
\]
For $q=\infty$ it follows from \prettyref{prop:discretesation_lower bound widths}
with $\left(k_{n}\right)_{n\in\N}\coloneqq\left(\left\lfloor c_{n}/2\right\rfloor \right)_{n\in\N}$,
where $\left\lfloor c_{n}/2\right\rfloor \asymp\mathcal{N}_{n}$,
\begin{align*}
\uAO_{\star}\left(\BB,L_{\nu}^{\infty}\right) & \geq\limsup_{n\to\infty}\frac{\log\left(d_{k_{n}}^{\star}\left(\BB,L_{\nu}^{\infty}\right)\right)}{\log(k_{n})}\gg\limsup_{n\to\infty}\frac{\log\left(2^{-n\rr}d_{k_{n}}^{\star}\left(\b_{p}^{c_{n}},\ell_{\infty}^{c_{n}}\right)\right)}{\log(k_{n})}\\
 & \gg\limsup_{n\to\infty}\frac{-\rr\log\left(2^{n}\right)}{\log\left(\mathcal{N}_{n}\right)}-\frac{\gamma\log\left(n_{k}\right)}{\log\left(n_{k}\right)}=\frac{-\rr}{\overline{\dim}_{M}\left(\nu\right)}-\gamma.
\end{align*}

\emph{ad (2):} For $q<\infty$, we consider $\a>0$ such that $\underline{F}_{\r}(\a)>0$,
otherwise there is nothing to show. By the definition, this implies
for all $\varepsilon\in\left(0,\underline{F}_{\r}(\a)\right)$ that
there exists $N\in\N$ such that for all $n\geq N$ 
\[
\frac{\log\left(5^{-m}\mathcal{N}_{\r,n}\left(\a\right)\right)}{n\log2}\geq\underline{F}_{\r}\left(\a\right)-\varepsilon.
\]
Hence, $5^{-m}\mathcal{N}_{\r,n}\left(\a\right)\geq2^{n\left(\underline{F}_{\r}\left(\a\right)-\varepsilon\right)}$
holds for $n$ large enough. For $k\in\N$, let
\[
n_{k}\coloneqq\left\lceil \frac{\log2k}{\left(\underline{F}_{\r}\left(\a\right)-\varepsilon\right)\log2}\right\rceil 
\]
giving $c_{\a,n_{k}}\geq\left\lfloor 5^{-m}\mathcal{N}_{\r,n_{k}}\left(\a\right)\right\rfloor \geq\left\lfloor 2^{n_{k}\left(\underline{F}_{\r}\left(\a\right)-\varepsilon\right)}\right\rfloor \geq2k.$
Using \prettyref{prop:discretesation_lower bound widths} then shows
\begin{align*}
\lAO_{\star}\left(\BB,L_{\nu}^{q}\right) & =\liminf_{k\to\infty}\frac{\log\left(d_{k}^{\star}\left(\BB,L_{\nu}^{q}\right)\right)}{\log(k)}\gg\liminf_{k\to\infty}\frac{\log\left(2^{-n_{k}\a/q}d_{k}^{\star}\left(\b_{p}^{c_{\a,n_{k}}},\ell_{q}^{c_{\a,n_{k}}}\right)\right)}{\log(k)}\\
 & =\liminf_{k\to\infty}\frac{-\a n_{k}\log(2)}{q\log(k)}+\frac{\log\left(d_{k}^{\star}\left(\b_{p}^{2k},\ell_{q}^{2k}\right)\right)}{\log(k)}\\
 & \geq\liminf_{k\to\infty}\frac{-\a\log(2k)}{q\log\left(k\right)\left(\underline{F}_{\r}\left(\a\right)-\varepsilon\right)}-\frac{\gamma\log\left(k\right)}{\log\left(k\right)}=\frac{-\a}{q\underline{F}_{\r}\left(\a\right)-\varepsilon}-\gamma.
\end{align*}
Finally, taking the supremum over all $\a>0$ and $\varepsilon\to0$
shows this case.

For $q=\infty$, we assume that $\underline{\dim}_{M}\left(\nu\right)>0$
to rule out the trivial case. As above, since $c_{n}\asymp\mathcal{N}_{n}$
this implies for all $\varepsilon>0$ and $n\in\N$ large enough 
\[
\frac{\log\left(c_{n}\right)}{n\log2}\geq\underline{\dim}_{M}\left(\nu\right)-\varepsilon,
\]
i.\,e\@. $c_{n}\geq2^{n\left(\underline{\dim}_{M}\left(\nu\right)-\varepsilon\right)}$
and setting $n_{k}\coloneqq\left\lceil \log\left(2k\right)/\left(\left(\underline{\dim}_{M}\nu-\varepsilon\right)\log2\right)\right\rceil $
for $k\in\N$ gives $c_{n_{k}}\geq2k$. Using \prettyref{prop:discretesation_lower bound widths}
shows
\begin{align*}
\lAO_{\star}\left(\BB,L_{\nu}^{\infty}\right) & \gg\liminf_{k\to\infty}\frac{\log\left(2^{-n_{k}\rr}d_{k}^{\star}\left(\b_{p}^{c_{n_{k}}},\ell_{\infty}^{c_{n_{k}}}\right)\right)}{\log(k)}\\
 & \geq\liminf_{k\to\infty}\frac{-\rr n_{k}\log(2)}{\log(k)}+\frac{\log\left(d_{k}^{\star}\left(\b_{p}^{2k},\ell_{\infty}^{2k}\right)\right)}{\log(k)}\\
 & \geq\liminf_{k\to\infty}\frac{-\rr\log\left(2k\right)}{\log\left(k\right)\left(\underline{\dim}_{M}\left(\nu\right)-\varepsilon\right)}-\frac{\gamma\log\left(k\right)}{\log\left(k\right)}=\frac{-\rr}{\underline{\dim}_{M}\left(\nu\right)-\varepsilon}-\gamma.
\end{align*}
Finally, taking $\epsilon\to0$ proves the second claim für $q=\infty$.
\end{proof}
\begin{proof}[Proof of lower bounds in \prettyref{thm:MAIN} and \prettyref{thm:MAIN_lower}]
 We combined \prettyref{lem:KolmogorovFiniteDim}, \prettyref{cor:GelfandFiniteDim}
and \prettyref{cor:LinearFiniteDim} with \prettyref{prop:AllLowerBounds}
(1) and the fact that $\overline{\mathcal{F}}_{\r}=s_{\r}$ to provide
us with all cases of the lower estimates of the upper approximation
orders in \prettyref{thm:MAIN} for the Kolmogorov and in \prettyref{rem:Main}
(4) for the Gel{\cprime}fand cases, and with \prettyref{prop:AllLowerBounds}
(2) for lower estimates of the lower approximation orders as needed
in \prettyref{thm:MAIN_lower}. The last step of the proof is to observe
\prettyref{eq:trivialObservation}.\printbibliography
\end{proof}

\end{document}